\documentclass[10pt,a4paper]{scrartcl} 
\usepackage{amsmath,amssymb}  
\usepackage{enumerate}% schicke Nummerierung
\usepackage{todonotes}
\usepackage{accents}
\allowdisplaybreaks
\usepackage[left=3cm,right=2cm,top=3cm,bottom=4cm]{geometry}

\usepackage{graphicx} % for Grafik-Einbindung

\usepackage[english]{babel}
\usepackage[T1]{fontenc}
\usepackage[utf8]{inputenc}  
  
\usepackage{amsthm}
\swapnumbers
\newtheorem{satz}{Theorem}[section]
\newtheorem{lemma}[satz]{Lemma}
\newtheorem{kor}[satz]{Corollary}
\theoremstyle{definition}

\newtheorem{bemerkung}[satz]{Remark}

\usepackage{color}
\definecolor{white}{rgb}{1,1,1}
\definecolor{darkred}{rgb}{0.3,0,0}
\definecolor{darkgreen}{rgb}{0,0.3,0}
\definecolor{darkblue}{rgb}{0,0,0.3}
\definecolor{pink}{rgb}{0.78,0.09,0.51}
\definecolor{purple}{rgb}{0.28,0.24,0.55}
\definecolor{orange}{rgb}{1,0.6,0.0}
\definecolor{grey}{rgb}{0.4,0.4,0.4}
\definecolor{aquamarine}{rgb}{0.4,0.8,0.65}

\newcommand{\N}{\mathbb{N}}

\newcommand{\R}{\mathbb{R}}

\newcommand{\norm}[2]{\left\Vert #2 \right\Vert_{L^{#1}(\Omega)}}

\newcommand{\wnorm}[1]{\left\Vert #1 \right\Vert_{W^{1,q}(\Omega)}}
\newcommand{\Norm}[2]{\left\Vert #2 \right\Vert_{#1}}
\newcommand{\tel}[1]{\frac{1}{#1}}
\newcommand{\diff}{\mathop{}\!\mathrm{d}}
\newcommand{\dt}{\frac{\diff}{\diff t}}
\newcommand{\io}{\int_\Omega}
\newcommand{\kreuz}{\nabla\cdot\left(\frac{u}{v}\nabla v\right)}
\newcommand{\euler}{\mathrm{e}}

\newcommand{\ubar}[1]{\underaccent{\bar}{#1}}
\newcommand{\uunten}{\ubar{u}}
\newcommand{\uoben}{\bar{u}}

\newcommand{\uvk}{\underaccent{\bar}{v}_k}
\newcommand{\Tmax}{T_{max}}
\newcommand{\Om}{\Omega}
\newcommand{\Ombar}{\overline{\Omega}}
\newcommand{\Lom}[1]{L^{#1}(\Omega)}
\newcommand{\f}[2]{\frac{#1}{#2}}
\newcommand{\set}[1]{\{#1\}}
\newcommand{\kl}[1]{\left(#1\right)}
\newcommand{\na}{\nabla}
\newcommand{\nn}{\nonumber}
\newcommand{\inttnullt}{\int_{(t-1)_+}^t}

\usepackage{tikz}
\usepackage{pgfpages}

\usepackage{newunicodechar}
\newunicodechar{ℝ}{\mathbb{R}}
\newunicodechar{ℕ}{\mathbb{N}}
\newunicodechar{Δ}{\Delta}
\newunicodechar{∞}{\infty}
\newunicodechar{ε}{\varepsilon} 
\newunicodechar{φ}{\varphi}
\newunicodechar{ζ}{\zeta}
\newunicodechar{α}{\alpha}
\newunicodechar{κ}{\kappa}
\newunicodechar{μ}{\mu}
\newunicodechar{⊂}{\subset}
\newunicodechar{ξ}{\xi}
\newunicodechar{λ}{\lambda}
\newunicodechar{χ}{\chi}

\title{Classical solutions to a logistic chemotaxis model with singular sensitivity and signal absorption}
\author{Elisa Lankeit\\
 {\small Institut f\"ur Mathematik, %}\\[-0.2cm]{\small
 Universit\"at Paderborn}\\[-0.2cm]
 {\small  Warburger Str. 100}\\[-0.2cm]
 {\small  33098 Paderborn, Germany}\\
{\small elankeit@math.uni-paderborn.de }
 \and
Johannes Lankeit\\%\footnote{jlankeit@math.uni-paderborn.de}\\
 {\small Institut f\"ur Mathematik, %}\\[-0.2cm]{\small
 Universit\"at Paderborn}\\[-0.2cm]
 {\small  Warburger Str. 100}\\[-0.2cm]
 {\small  33098 Paderborn, Germany} \\
{\small jlankeit@math.uni-paderborn.de}
}
\begin{document}
\maketitle 

\begin{abstract}
\noindent Assuming that $0<\chi<\sqrt{\frac{2}n}$, $\kappa\ge 0$ and $\mu>\frac{n-2}{n}$, we prove global existence of classical solutions to a chemotaxis system slightly generalizing 
\[ 
\begin{cases} u_t = \Delta u - \chi \nabla\cdot \big( \frac{u}{v} \nabla v \big) + \kappa u -\mu u^2 \\
 v_t = \Delta v - u v \end{cases} 
\]
in a bounded domain $\Omega \subset \mathbb{R}^n$, with homogeneous Neumann boundary conditions and for widely arbitrary positive initial data.  
In the spatially one-dimensional setting, we prove global existence and, moreover, boundedness of the solution for any $χ>0$, $μ>0$, $κ\ge 0$.\\

\end{abstract}

\section{Introduction}

In chemotaxis systems with singular sensitivity, signal evolution being gouverned by a consumptive equation can make even the global existence analysis challenging. One particular system of the mentioned type is the following, 
\begin{align}\label{intro:sys1}
 u_t&=Δ u - \chi\nabla\cdot\left(\frac{u}{v}\nabla v\right)+f(u),\\
 v_t&=Δ v - uv\nn,
\end{align}
a slightly more general form of which with $f\equiv 0$ has been introduced in \cite{kellersegel_trav} by Keller and Segel in order to capture the behaviour of \textit{Escherichia coli} (with population density $u$) set on a substrate containing varying amounts of oxygen and an energy source. Bacteria of this species are chemotactically active and partially direct their movement toward higher concentration of the ``signal'' substance (%which can be either of oxygen or the energy source and 
whose concentration is given by $v$) -- in accordance with the Weber--Fechner law of stimulus perception (see \cite{kellersegel_trav}), the direction and intensity of this movement are given by the gradient of the logarithm of the signal concentration (with a proportionality constant $χ>0$). 

Intuitively, we can imagine that the second equation of \eqref{intro:sys1} pushes $v$ toward zero, whereas in the first equation $v$ being small is exactly what boosts %enhances, intensifies 
the (destabilizing) effects of the cross-diffusive term modelling the chemotaxis.

This is different from the %situation in the 
more commonly studied chemotaxis systems with signal production, like 
%\cite{Zhao_sensitivity}%GE,bdd; condition on k in v_t=kΔv -v+u 
\begin{align}\label{intro:sys-prod}
 u_t&=Δ u - \chi\nabla\cdot\left(\frac{u}{v}\nabla v\right)+f(u),\\
 v_t&=Δ v - v + u\nn,
\end{align}
(see \cite{Biler99, win-ge, fujie, masaakitomomi,newapproach, Zhao_sensitivity,wang-li-yu}), or from the system 
\begin{align} 
\label{intro:sys-classical}  u_t&=Δ u - \chi\nabla\cdot\left(u\nabla v\right)+f(u),\\
 v_t&=Δ v - v + u\nn,
\end{align}
the latter for $f\equiv 0$ also being known as ``the'' classical Keller--Segel system (\cite{kellersegel_slime,horstmann})%; for an overview of the rich literature dealing with this system and some of its relatives, consult the surveys \cite{...}
. Here the signal is assumed to be produced by the bacteria themselves and the corresponding source term $+u$ in the second equation tends to keep $v$ away from the singularity in the first equation of \eqref{intro:sys-prod}; and \eqref{intro:sys-classical} does not contain such a singularity in the sensitivity function at all. 

As to \eqref{intro:sys-prod} with $f\equiv 0$, it has been shown that the form of taxis inhibition at large densities of the signal 
can prevent blow-up of solutions \cite{Biler99,win-ge} and even lead to their global boundedness \cite{fujie,newapproach}, if $\chi<\chi_0(n)$, where $\chi_0(2)>1.015$ (\cite{newapproach}) and $\chi_0(n)=\sqrt{\f2n}$ for $n\ge 3$ \cite{fujie,masaakitomomi}. If $n\ge 3$ and $\chi>\f{2n}{n-2}$, a corresponding parabolic--elliptic analogue is known to admit solutions exploding in finite time \cite{nagaisenba}. 
Global existence of solutions to the parabolic--parabolic system in parts of the remaining range for $\chi$ was proven in cases where one component diffuses fast if compared to the other (\cite{fujiesenba1,fujiesenba2}) or for certain weaker concepts of solutions, at least excluding Dirac-type singularities (\cite{winkler_singular,stinnerwinkler,lanwin}). Also the coupling to a fluid has been investigated in \cite{BLM2}, yielding global classical solutions whenever $χ<\sqrt{\f2n}$. 
For a multitude of results concerning \eqref{intro:sys-classical} and some of its variants, we refer to the surveys \cite{BBTW,horstmann,hillen-painter}.

That both systems \eqref{intro:sys-prod} and \eqref{intro:sys-classical} do not involve the particular difficulty of combining consumption with singular sensitivity may be the main reason why the knowledge concerning existence of solutions to \eqref{intro:sys1} (for the moment we remain with the case of $f\equiv0$) is much sparser: 

It has long been known that travelling wave solutions (see also \cite{TWS-survey}) exist in a one-dimensional setting; this observation goes back to the work \cite{kellersegel_trav} by Keller and Segel. But only recently, general existence results were obtained: In \cite{wangxiangyu_asymptotic}, it was proven that (for $\Om\in\set{ℝ^2,ℝ^3}$) solutions exist globally whenever a strong smallness condition (involving up to second derivatives) is imposed on the initial data; %p. 2236: κ_0\calF_0 has to be small, F_0 contains H^2 sorms 
%without smallness condition on the initial data %in spaces like H^2
%was obtained: In \cite{win_ct_sing_abs} 
without this condition (and for bounded domains $\Om\subset ℝ^2$), in 
\cite{win_ct_sing_abs}
the existence of generalized solutions was proven 
where these solutions are classical after some waiting time if the initial bacterial mass $\io u_0$ is small \cite{win_ct_sing_abs_eventual}. These results can be carried over also to the situation where \eqref{intro:sys1} is coupled to a Stokes fluid  (\cite{YilongWang2016}, \cite{black}). In higher dimensional situations, however, even existence results (without small-data conditions) seem to be elusive. 
In \cite{win_ct_sing_abs_renormalized}, renormalized solutions could be proven to exist in a radially symmetric setting.  

%Sufficiently severe modifications
On the other hand, certain modifications of the system, like introducing stronger, porous-medium type diffusion for the first component (\cite{locallybounded}), or replacing $u$ in the chemotaxis term by functions similar to $u^{α}$ with smaller exponents $α$ (\cite{dongmei}), can enforce global existence of classical solutions. 

In models without singular sensitivity (that is, for example, \eqref{intro:sys-classical} instead of \eqref{intro:sys-prod})%($χ\nabla\cdot(u\nabla v)$ instead of $χ\nabla\cdot(\f uv\na v)$)
, another such mechanism whose presence is known to help avoid blow-up of solutions, in some cases at least, is given by logistic source terms $f(u)=κu-μu^2$. The biological reasoning behind addition of these terms is a desire to capture effects of population growth, including death effects due to overpopulation. Often, the growth parameter $κ$ plays no important role in the analysis and could be set to $0$ or sometimes even negative (in modelling starving populations); the parameter $μ$, on the other hand, can be decisive in the proofs. Of particular interest are very small, but nevertheless positive, values of $μ$, because population growth occurs on a larger timescale than diffusive or chemotactical motion.  

While logistic sources by no means make the dynamical properties of system \eqref{intro:sys-classical} uninteresting (cf. the numerical experiments of \cite{painterhillen-chaos}, findings on exponential attractors in \cite{otym} or the results on transient growth phenomena in \cite{win_exceed_ell,lan_exceed,win_exceed_par}), they work against the aggregative tendencies of the cross-diffusive term: In two-dimensional domains, system \eqref{intro:sys-classical} with $f(u)=κu-μu^2$ admits global solutions \cite{osaki-yagi} for arbitrary positive $μ$; in higher dimensions, $μ$ must be sufficiently large for the global existence proofs to be applicable (\cite{mygross}), but at least global weak solutions are known to exist for any $μ>0$ (\cite{eventualsmoothness}). An explicit largeness condition on $μ$ for global existence in the parabolic--elliptic counterpart of \eqref{intro:sys-classical} is $μ>\f{n-2}n$ (\cite{tellowin}), which we mention for comparison with 
condition \eqref{conditiononmuchi} of Theorem \ref{thm:main:ge}. Logistic sources 
play a similar role in attraction-repulsion chemotaxis systems \cite{attraction-repulsion} or chemotaxis systems with fractional diffusion \cite{jan-rafael}, and have also been studied with smaller exponents \cite{nakaguchi_osaki}, \cite{giuseppe}, or nonconstant parameters \cite{tahir}; for results on convergence rates see \cite{he_zheng_convergence}.  

Also in consumption models (like \eqref{intro:sys1} without singular sensitivity) that have been studied in the context of chemotaxis-fluid models (see e.g. \cite{lorz}, \cite{win_ctns}, \cite[Sec. 4.1.1]{BBTW}), but also without fluid (e.g. \cite{Tao_consumption}), presence of $f(u)=κu-μu^2$ can help in the derivation of suitable a priori estimates and finally ensure global existence of classical solutions or of weak solutions eventually regularizing (\cite{joh_fluid,yulanjoh}). 

It therefore stands to reason that also in systems with singular sensitivity, like \eqref{intro:sys1} or \eqref{intro:sys-prod}, logistic source terms can help ensuring global existence. Indeed, this can be the case, as shown for two-dimensional domains in \cite{blowupprev} and \cite{aotym,Zhao2016} for the parabolic--elliptic analogue of \eqref{intro:sys-prod} and \eqref{intro:sys-prod} itself, respectively. The boundedness results in \cite{blowupprev} and \cite{Zhao2016} feature a \textit{largeness} assumption on $κ$, indicating that somehow a large bacterial mass makes the signal concentration avoid the singularity in the sensitivity (cf. also \cite[Lemma 2.2]{fujie}). As noted above, the most striking difference between \eqref{intro:sys-prod} and \eqref{intro:sys1} is that just this mechanism is absent in \eqref{intro:sys1}. In \cite{BLM2}, a similar lack of a positive global lower bound for the second component (there caused by transport terms arising from a coupling of \eqref{intro:sys-prod} to a (
Navier--)
Stokes fluid) seemed to make it impossible to find a boundedness result. 

But even if not boundedness, can we at least guarantee global existence? Can we, in doing so, possibly even surpass the restriction on the dimension used in both the logistic results \cite{blowupprev,Zhao2016} concerning \eqref{intro:sys-prod} and in \cite{win_ct_sing_abs} dealing with a source term free variant of \eqref{intro:sys1}?  

We will attempt to pursue these questions, and, more precisely, treat the system 
\begin{align}\label{sys}\begin{array}{rll}
 u_t&=Δ u -\chi\nabla\cdot\left(\frac{u}{v}\nabla v\right)+f(u),\\%&x\in\Omega, &t>0,\\
 v_t&=Δ v -uv,\\% &x\in\Omega, ~&t>0,\\
 \partial_{\nu}u&=\partial_{\nu}v=0,\\%&x\in\partial\Omega, ~&t>0,\\
 u(\cdot,0)&=u_0, \quad v(\cdot,0)=v_0,%&x\in\Omega.&
 \end{array}
\end{align}
posed for positive time and in a bounded spatial domain $\Om\subset ℝ^n$, $n\in ℕ$, with smooth boundary for initial data 
\begin{equation}\label{cond:init}
 u_0 \in C^0(\Ombar), \qquad v_0\in %\bigcup_{q>n}
 W^{1,∞}(\Om),\;\;\; u_0\ge 0,\; v_0>0 \text{ in } \Ombar. 
\end{equation}
We will assume that there are $κ\ge 0$, $μ>0$ such that 
\begin{equation}\label{cond:f}
 f\in C^1(ℝ), \quad f(0)=0,\quad f(u)\le κu-μu^{2} \text{ for } u\ge 0. 
\end{equation}
\begin{bemerkung}\label{rem:f}
If $f\in C^1(ℝ)$, $f(0)=0$ and $\limsup_{u\to\infty} u^{-2}f(u) \le -μ<0$, then there is $u_*>0$ such that $f(u)\le -μu^2$ for $u>u_*$ and with some $κ\geq \sup_{u\in (0,u_*]} \f{f(u)+μu^{2}}u$ (which exists due to $f(0)=0$ and finiteness of $f'(0)$), condition \eqref{cond:f} is certainly satisfied. 
\end{bemerkung} 

We will, moreover, require that $χ\in\kl{0,\sqrt{\f2n}}$, similar to the above-mentioned results on \eqref{intro:sys-prod}. 

Under these assumptions, we will prove global existence of solutions for sufficiently large values of $μ$, and in particular, notably, for any positive $μ$, if $n=2$: 

\begin{satz}\label{thm:main:ge}
Let $n\in ℕ$, $n\ge2$, and let $\Om\subset ℝ^n$ be a bounded domain with smooth boundary. Assume that 
\begin{equation}\label{conditiononmuchi}
 0<χ<\sqrt{\f 2n}, \qquad μ>\f{n-2}n, 
\end{equation}
$κ\ge 0$, and that $f$ satisfies \eqref{cond:f}. Then for any initial data $(u_0,v_0)$ as in \eqref{cond:init} there is a global classical solution $(u,v)$ to \eqref{sys}. 
\end{satz}
\begin{bemerkung}
 In fact, \eqref{sys} with $f(u)=ru-μu^{α}$ has very recently been treated by Zhao and Zheng in \cite{veryrecent}, where they proved global existence if $α>1+\f{n}2$. Hence, in the $2$-dimensional setting, Theorem \ref{thm:main:ge}, which covers $α=2=1+\f n2$, can be seen as a natural extension of their theorem to the boundary case not encompassed in \cite{veryrecent}; in higher dimensions, the condition on $α$ in Theorem \ref{thm:main:ge} is much weaker than in \cite{veryrecent}. On the other hand, an additional restriction on $χ$ is needed. These differences in the results stem from the fact that the proofs are based on different main ideas (for further comments, see below).
\end{bemerkung}

Our second main result is concerned with the question of boundedness, which we can assure if $\Om$ is a one-dimensional domain: 

\begin{satz}\label{thm:main:bd1d}
 Let $\Om\subset ℝ$ be an open, bounded interval. Let $κ\ge 0$, $μ>0$ and $χ>0$. Assume that $(u_0,v_0)$ satisfies \eqref{cond:init} and $f$ fulfils \eqref{cond:f}. Then \eqref{sys} has a global solution which is, furthermore, bounded.
% In addition to the assumptions of Theorem \ref{thm:main:ge} let $n=1$. Then the solution $(u,v)$ from Theorem \ref{thm:main:ge} is bounded. 
\end{satz}
Global existence of solutions in the one-dimensional case may not be very surprising, since it has been proven even for $f\equiv 0$ in \cite{dongmei} and the system with logistic source in \cite{veryrecent}. However, neither of these results is concerned with their boundedness (as one can see from, e.g., \cite[(4.3)]{veryrecent}).

After introducing a local existence result (Theorem \ref{satz:lok}) in Section \ref{kap:lokal}, whose proof we will detail in Appendix \ref{sec:locex-proof}, and ensuring some simple properties of the solution, like nonnegativity of both components, mass conservation for $u$, the spatio-temporal estimate resulting from presence of the logistic source, and boundedness of $v$ in Section \ref{kap:eigenschaften}, in Section \ref{sec:vpos}, we work with a transformed, non-singular system arising from the substitution $w:=-\log \f{v}{\Norm{\Lom\infty}{v_0}}$ in order to derive a positive lower estimate for $v$. For the global existence proof we rely on an iterative procedure (in the proof of Lemma \ref{lm:bd:u:nhalbe}), whose steps are based on Lemma \ref{lm:gradv} and thereby on semigroup estimates for the Neumann heat semigroup, and whose starting point is an estimate for $\Norm{\Lom p}{u(\cdot,t)}$ for some $p>\f n2$. We obtain the latter in Lemma \ref{lm:dasistthm1} from the fact that for suitably chosen $p>0$ 
and 
$r>0$ 
the functional 
\[
 \io u^pv^{-r} 
\]
 satisfies a differential inequality of the form $\f{d}{dt} \io u^pv^{-r}\le C \io u^pv^{-r}$ (proof of Lemma \ref{upvr}), as long as $μ$ and $χ$ fulfil the conditions of Theorem \ref{thm:main:ge}. Use of this functional also marks one of the main differences to the approach in \cite{veryrecent}, where the estimates essentially originate from the absorptive term $-\int_0^T\io u^{α}$, which arises from the source function $f(u)=ru-μu^{α}$ used there with larger exponents $α$.
In Section \ref{kap:1d}, we will restrict ourselves to the one-dimensional setting and, again combining the transformed system with semigroup estimates, prove Theorem \ref{thm:main:bd1d}. 
In Appendix \ref{sec:comparison}, finally, we state and prove a comparison theorem, a useful and, in fact, often-used tool, of which we, nevertheless, did not find a version entirely suitable for application to \eqref{sys} in the literature.

\section{Local existence of a classical solution}\label{kap:lokal}

Before we can investigate global existence or qualitative properties of classical solutions, we have to consider the local-in-time existence and uniqueness of solutions to \eqref{sys}.

% \begin{definition}
%  Eine klassische Lösung von \ref{sys} ist ein Funktionenpaar $(u,v)\in $ EINFÜGEN, das alle Gleichungen in \ref{sys} erfüllt. 
% \end{definition}
Local existence of classical solutions often is obtained from the following useful lemma taken from the survey \cite{BBTW}: 
\begin{lemma}\label{satz:survey}
Let $n\ge 1$ and $\Om\subset ℝ^n$ be a bounded domain with smooth boundary, and let $q>n$. For some $\omega\in(0,1)$ let $S\in C_{\text{loc}}^{1+\omega}(\overline{\Omega}\times[0,\infty)\times\R^2)$, $h\in C^{1-}(\overline{\Omega}\times[0,\infty)\times\R^2)$ and $g\in C_{\text{loc}}^{1-}(\overline{\Omega}\times[0,\infty)\times\R^2)$ with $h(x,t,0,v)\geq 0$ for all $(x,t,v)\in \overline{\Omega}\times[0,\infty)^2$ and $g(x,t,u,0)\geq 0$ for all $(x,t,u)\in \overline{\Omega}\times[0,\infty)^2$.\\
  Then for all non-negative $u_0\in C^0(\Ombar)$ and $v_0\in W^{1,q}(\Om)$ there exist $\Tmax\in(0,\infty]$ and a uniquely determined pair of non-negative functions 
 \begin{align*}
 u&\in C^0(\overline{\Omega}\times[0,\Tmax))\cap C^{2,1}(\overline{\Omega}\times(0,\Tmax))\quad\text{and }\\
 v&\in C^0(\overline{\Omega}\times[0,\Tmax))\cap C^{2,1}(\overline{\Omega}\times(0,\Tmax))\cap L_{\text{loc}}^{\infty}([0,\Tmax);W^{1,q}(\Omega)),
 \end{align*}
 such that for $(u,v)$ in $\Omega\times(0,\Tmax)$ we have:
 \begin{align*}
  \begin{array}{lr}
   u_t=Δ u -\nabla\cdot\left(uS(x,t,u,v)\nabla v\right)+h(x,t,u,v),&x\in\Omega, ~t>0,\\
   v_t=Δ v-v+g(x,t,u,v),&x\in\Omega, ~t>0,\\
   \frac{\partial u}{\partial \nu}=\frac{\partial v}{\partial \nu}=0,&x\in\partial\Omega, ~t>0,\\
   u(x,0)=u_0(x), v(x,0)=v_0(x),&x\in\Omega,
  \end{array}
 \end{align*}
 and $\Tmax=\infty$ or $\norm{\infty}{u(\cdot,t)}+\wnorm{v(\cdot,t)}\rightarrow \infty$ for $t\nearrow \Tmax$.
\end{lemma}
\begin{proof}
 A detailed proof can be found in \cite[Lemma 3.1]{BBTW}. There by Banach's fixed point theorem the existence of mild solutions on an interval $[0,T)$ is shown, where $T$ depends on the initial data $u_0$ and $v_0$. Bootstrap arguments then provide the required regularity. That the solution can be extended up to some maximal $\Tmax$ with the desired property follows from the fact that $T$ depends on $\norm{\infty}{u_0}$ and $\wnorm{v_0}$ only.
\end{proof}

Attempting to bring system \eqref{sys} into the shape assumed in Lemma \ref{satz:survey} results in the choices
\begin{align}\label{susw}
 S(x,t,u,v)=\frac{\chi}{v}, &&  h(x,t,u,v)=f(u), && g(x,t,u,v)=v-uv.
% 
% \begin{array}{ll}
%  S(x,t,u,v)&=\frac{\chi}{v}\\
%  h(x,t,u,v)&=f(u)\\
%  g(x,t,u,v)&=v-uv.
%  \end{array}
\end{align}
Lemma \ref{satz:survey}, however, does not yield existence of solutions to \eqref{sys}, where $S$ is singular at $v=0$ and $h$ is not known to be Lipschitz continuous. With adjustments akin to those in the proof of \cite[Thm. 2.3]{newapproach} we can, nevertheless, find a way to employ Lemma \ref{satz:survey}.

\begin{satz}[Local existence and extensibility]\label{satz:lok}
 Let $n\geq 1$, $\Omega\subseteq\R^n$ a bounded, smooth domain and $q>n$. Then for all non-negative functions $u_0\in C^0(\overline{\Omega})$ and functions $v_0\in W^{1,\infty}(\Omega)$, positive throughout $\Ombar$, there are $\Tmax\in (0,\infty]$ and a unique pair of functions  $(u,v)$ with 
 \begin{align*}
 u&\in C^0(\overline{\Omega}\times[0,\Tmax))\cap C^{2,1}(\overline{\Omega}\times(0,\Tmax))\quad\text{and }\\
 v&\in C^0(\overline{\Omega}\times[0,\Tmax))\cap C^{2,1}(\overline{\Omega}\times(0,\Tmax))\cap L_{\text{loc}}^{\infty}([0,\Tmax);W^{1,q}(\Omega)),
 \end{align*}
solving \eqref{sys} in the classical sense on  $\Omega\times[0,\Tmax)$ and with 
\begin{equation}\label{extensibilitycrit}
\Tmax=\infty \text{ or } \norm{\infty}{u(\cdot,t)}+\wnorm{v(\cdot,t)}\rightarrow \infty\text{ as }t\nearrow \Tmax.
\end{equation}
\end{satz}
\begin{proof}
We postpone the somewhat technical details of the proof to Appendix \ref{sec:locex-proof}.
\end{proof}

%Having ensured existence of a unique solution to \eqref{sys} up to some $\Tmax$, we now will consider $u$ and $v$ in more detail and begin by proving basic properties of these functions, which can be obtained from the system of differential equations.

\section{Basic properties of the solution}\label{kap:eigenschaften}
In this section we will collect some basic properties of solutions to \eqref{sys} in $\Om\times[0,\Tmax)$. We will assume that $f$ is a given function fulfilling \eqref{cond:f} for $κ$, $μ$, which we suppose to be arbitrary numbers $κ\ge 0$, $μ>0$ and that $χ>0$ is arbitrary, unless otherwise specified. We will always use $(u,v)$ to denote the unique solution to \eqref{sys} in $\Om\times[0,\Tmax)$ with $\Tmax$ as provided by Theorem \ref{satz:lok} for given, fixed initial data $(u_0,v_0)$ as in Theorem \ref{satz:lok}. The properties we collect here will be fundamental for showing globality of solutions in Chapter \ref{kap:global} under additional conditions. 

First let us note the unsurprising, but nevertheless important, fact that positive initial data ensure positivity of the solution. 
\begin{lemma}\label{geq0}
 We have $u\geq 0$ and $v>0$ on $\Omega\times [0,\Tmax)$.
\end{lemma}
\begin{proof}
 Noting that due to Remark \ref{bem:vs} the comparison theorem \ref{vs} can be applied to the equations of the system \eqref{sys}, from comparison with the subsolution $\ubar{u}\equiv 0$ we immediately obtain nonnegativity of $u$. That $v$ has to be positive in any finite time interval $[0,\hat{t}]$ has already been shown in Remark \ref{bem:vs}.
\end{proof}

System \eqref{sys} does not enjoy a mass conservation property (i.e. $\norm{1}{u}=$ const), in contrast to e.g. the system in \cite{locallybounded}. Nevertheless, boundedness of $u$ in $\Lom1$ can be shown easily: 
\begin{lemma}\label{u}
 There is some $m>0$ such that 
 \begin{equation*}
  \norm{1}{u(\cdot,t)}\leq m \quad\text{for all } t\in[0,\Tmax).
 \end{equation*}
\end{lemma}
\begin{proof}
 Due to the homogeneous Neumann boundary conditions and \eqref{cond:f}, the time derivative of $\io u$ satisfies 
 \begin{align*}
  \dt \int_{\Omega} u %&=\int_\Omega u_t=\io Δ u - \chi \io \nabla \cdot\left(\frac{u}{v}\nabla v\right)+\kappa\io u -\mu \io u^2\\
		      &\le \kappa\io u -\mu \io u^{2}\quad\text{on  }(0,\Tmax).
 \end{align*}
 An application of Hölder's inequality hence results in 
%  wegen des Satzes von Gauß.\\
%  Aufgrand der Hölderungleichung gilt
%  \begin{align*}
%   \int_\Omega u \leq |\Omega|^{\tel{2}}\left(\int_\Omega u^2\right)^{\tel{2}}\quad\text{on  }(0,\Tmax),
%  \end{align*} 
%  also 
%  \begin{align*}
%   \left(\io u\right)^2\leq |\Omega| \io u^2\quad\text{on  }(0,\Tmax) .
%  \end{align*}
%  Damit folgt 
 \begin{align*}
  \dt \io u \leq \kappa \io u - \mu|\Omega|^{-1} \left(\io u\right)^{2}\quad\text{on  }(0,\Tmax),
 \end{align*}
 and an ODI comparison argument yields the result. 
\end{proof}

The logistic term in \eqref{sys} entails a spatio-temporal estimate for $u^{2}$, which will prove useful later on (see Section \ref{kap:1d}).
\begin{lemma}\label{u2}
 There is a constant $C>0$ such that
 \begin{align*}
  \inttnullt \!\!\io u^{2} \leq C 
 \end{align*}
for all $t\in [0,\Tmax)$.
\end{lemma}
\begin{proof} 
Inserting \eqref{cond:f} into the first equation of \eqref{sys}, solving for $u^{2}$ and integrating over $\Om\times((t-1)_+,t)$, we obtain 
 \begin{align*}
  \inttnullt \!\!\io u^{2} %=&\tel{\mu}\left( \int_t^{t+1}\!\!\io Δ u - \chi\int_t^{t+1}\!\!\io \kreuz +\kappa \int_t^{t+1}\!\!\io u - \int_t^{t+1}\!\!\io u_t  \right)\\
  \le & \tel{\mu} \left( \kappa \inttnullt\io u -\io u(\cdot, t) + \io u(\cdot, t_0)\right)%\\
  \leq \frac{(\kappa+1)}{\mu}m=:C,
 \end{align*}
where $m$ is taken from Lemma \ref{u}.
\end{proof}

One of the main tools for making use of such spatio-temporal estimates is given by the following simple lemma: 
\begin{lemma}\label{odi-spatiotemporal}
 For some $T\in (0,\infty]$ let $y\in C^1((0,T))\cap C^0([0,T))$, $h\in C^0([0,T))$, $C>0$, $a>0$ satisfy 
\[
 y'(t)+ay(t) \le h(t), \qquad \int_{(t-1)_+}^t h(s)\diff s\le C
\]
for all $t\in(0,T)$. Then $y\le y(0)+\f{C}{1-e^{-a}}$ throughout $(0,T)$.
\end{lemma}
\begin{proof}
 From the variation-of-constants formula, we can conclude that
\[
 y(t)\le y(0)e^{-at} + \int_0^t e^{-as} h(t-s) \diff s,\quad  t\in(0,T).
\]
If, for convenience of notation, we let $h(s):=0$ for $s\le 0$, we have $\int_k^{k+1} h(t-s) \diff s\le C$ for any $k\in ℕ_0$ and $t\in(0,T)$, and can estimate 
\[
 y(t)\le y(0) + \sum_{k=0}^{\infty} \int_k^{k+1} e^{-as} h(t-s) \diff s \le y(0) + C \sum_{k=0}^\infty e^{-ak}, \qquad t\in(0,T). \qedhere
\]
\end{proof}

Also for $v$ the differential equation directly entails some decay properties: 
\begin{lemma}\label{v}
 For every $p\in [1,\infty)$ the map $(0,\Tmax)\ni t\mapsto\norm{p}{v(\cdot,t)}^p$ is monotone decreasing. In particular, $\norm{p}{v(\cdot,t)}\leq\norm{p}{v_0}$ for all $t\in[0,\Tmax)$.
\end{lemma}
\begin{proof}
 If we consider the derivative of said mapping, integration by parts results in 
 \begin{align*}
  \dt\io v^p&=p\io v^{p-1}v_t=p\io v^{p-1}Δ v-p\io uv^p \\&=-p(p-1)\io v^{p-2}|\nabla v|^2-p\io u v^p\leq 0\quad\text{on  }(0,\Tmax)
 \end{align*}
due to Lemma \ref{geq0}.
\end{proof}
Also $\norm{\infty}{v(\cdot,t)}$ can be controlled by the size of $v_0$: 
\begin{lemma}\label{vinfty}
 We have $\norm{\infty}{v(\cdot,t)}\leq \norm{\infty}{v_0}$ for every $t\in[0,\Tmax)$.
\end{lemma}
\begin{proof}
 Defining $\bar{v}(x,t):=\norm{\infty}{v_0}$, due to the nonnegativity of $u$ we see 
 \begin{align*}
  \bar{v}_t=0&\geq -u\bar{v}=Δ \bar{v}-u\bar{v}\quad\text{in }\Omega\times(0,\Tmax),\\
  \bar{v}(x,0)&\geq v_0(x)\quad\text{for all } x\in\Omega \quad\text{as well as}\\
  \partial_{\nu}\bar{v}(x,t)&=0\quad\text{for all } x\in\Omega, t\in(0,\Tmax).
 \end{align*}
Hence the comparison theorem (Theorem \ref{vs}) implies $v\leq \norm{\infty}{v_0}$ on  $\Omega\times [0,\Tmax)$ and thus in particular $\norm{\infty}{v(\cdot,t)}\leq\norm{\infty}{v_0}$ for every $t\in [0,\Tmax)$.
\end{proof}

\section{A non-singular system}\label{sec:vpos}
The singularity in the PDE system is bothersome in further calculations. It would be helpful to estimate $v$ from below by a positive constant in order to ensure the existence of some $d>0$ with $\tel{v}\leq d$. In contrast to the system \eqref{intro:sys-prod}, which has been considered in \cite{winkler_singular}, it will not be obvious that some $c>0$ with $v\geq c>0$ in $\Omega\times(0,\Tmax)$ exists. Under certain conditions, however, it is possible to show an estimate of this kind locally in time. As in \cite{locallybounded}, we introduce  $w:=-\log\left(\frac{v}{\norm{\infty}{v_0}}\right)$ and then prove time-local boundedness of $w$. The advantage of this approach is that we obtain a new system without singularity. Keeping the notation and general assumptions of Section \ref{kap:eigenschaften}, throughout this section we let $w$ be as just specified.

\begin{lemma}\label{w1}
 For $w:=-\log\left(\frac{v}{\norm{\infty}{v_0}}\right)$ we have $w\geq 0$ and $w_t=Δ w-|\nabla w|^2+u$ on  $\Omega\times (0,\Tmax)$.  
\end{lemma}
\begin{proof}
 According to Lemma \ref{vinfty}, $v\leq \norm{\infty}{v_0}$, hence $\frac{v}{\norm{\infty}{v_0}}\leq 1$ and thus $w\geq 0$. 
 Moreover, on  $\Omega\times(0,\Tmax)$,
 \begin{align*}
  w_t&=-\frac{\norm{\infty}{v_0}v_t}{v\norm{\infty}{v_0}}=-\frac{v_t}{v},%\\
&  \nabla w &= -\frac{\norm{\infty}{v_0}\nabla v}{v\norm{\infty}{v_0}}=-\frac{\nabla v}{v},
 \end{align*}
 which entails 
 \begin{align*}
  Δ w&=\nabla \cdot \nabla w=\nabla\cdot\left(-\frac{\nabla v}{v}\right)=-\frac{Δ v}{v}+\frac{|\nabla v|^2}{v^2}=-\frac{Δ v}{v}+|\nabla w|^2.
 \end{align*}
 Together with $v_t=Δ v-uv$ this proves 
 \begin{align*}
  w_t=-\frac{v_t}{v}=-\frac{Δ v}{v}+u=Δ w-|\nabla w|^2+u
 \end{align*}
 on  $\Omega\times(0,\Tmax)$.
\end{proof}

Accordingly, the pair $(u,w)$ solves the PDE system 
\begin{align}\label{sysw}
 \begin{array}{rlll}
 u_t&=Δ u +\chi\nabla\cdot(u\nabla w)+f(u)&\text{in }\Omega\times (0,\Tmax),\\
 w_t&=Δ w -|\nabla w|^2+u&\text{in }\Omega\times (0,\Tmax),\\
 \partial_{\nu}u&=\partial_{\nu}w=0&\text{in }\partial\Omega\times (0,\Tmax),\\
 u(\cdot,0)&=u_0, \quad w(\cdot,0)=w_0:=-\log\left(\frac{v_0}{\norm{\infty}{v_0}}\right)&\text{in }\Omega.&
 \end{array}
\end{align}
For its solution $(u,w)$ we can show the following proposition: 
\begin{lemma}\label{w}
 If for $T\leq \Tmax$ with $T<\infty$ there are a constant $C=C(T)>0$ and some $p\ge1$, satisfying $p>\frac{n}{2}$, with $\norm{p}{u(\cdot,t)}\leq C$ on  $(0,T)$, then $w$ is bounded on  $\Omega\times(0,T)$.
 %Wenn ein $C>0$ and $p>\frac{n}{2}$ mit $\norm{p}{u}\leq C$ existieren, ist $w$ on  jedem endlichen Zeitintervall beschränkt.
 %$w$ ist beschränkt, wenn ein $C>0$ and $p>\frac{n}{2}$ mit $\norm{p}{u}\leq C$ existieren.
\end{lemma}
\begin{proof}
 Let $T\in(0,\Tmax]$, $T<\infty$ and $p>\frac{n}{2}$ with $\norm{p}{u(\cdot,t)}\leq C$ for all $t\in(0,T)$ and some $C=C(T)>0$. 
 By the variation-of-constants formula, $w$ can be represented as 
 \begin{align*}
  w(\cdot,t)&=\euler^{tΔ}w_0+\int_0^t\euler^{(t-s)Δ}\left(u(\cdot,s)-|\nabla w(\cdot,s)|^2\right)\diff s\\
  &\leq \euler^{tΔ}w_0+\int_0^t\euler^{(t-s)Δ}u(\cdot,s)\diff s\quad\text{for } t\in(0,T),
 \end{align*}
 for  $-|\nabla w(\cdot,s)|^2\leq 0$ immediately implies $\euler^{(t-s)Δ}\left(-|\nabla w(\cdot,s)|^2\right)\leq0$. 
 With this representation and semigroup estimates as in \cite[Lemma 1.3 (i)]{lplq}, we obtain $c_1>0$ such that for all $t\in(0,T)$
 \begin{align*}
  \norm{\infty}{w(\cdot,t)}&\leq \norm{\infty}{\euler^{tΔ}w_0}+\int_0^t\norm{\infty}{\euler^{(t-s)Δ}u(\cdot,s)}\diff s\\
  &\leq \norm{\infty}{w_0}+\int_0^t c_1\left(1+(t-s)^{-\frac{n}{2p}}\right)\norm{p}{u(\cdot,s)}\diff s\\
  &\leq \norm{\infty}{w_0}+c_1C\int_0^T \left(1+(t-s)^{-\frac{n}{2p}}\right)\diff s <\infty,
 \end{align*}
 since $p>\frac{n}{2}$ implies that $-\frac{n}{2p}>-1$ and thus finiteness of the integral. 
\end{proof}

In order to return to $v$, we state the following: 
\begin{lemma}\label{vpos}
 If for $T\leq \Tmax$, $T<\infty$ there are a constant $C=C(T)>0$ and some $p>\frac{n}{2}$, $p\ge 1$, with $\norm{p}{u(\cdot,t)}\leq C$ on  $(0,T)$, there are  $d=d(T)>0$, such that $v\geq d$ and in particular $\tel{v}\leq\tel{d}$ on  $\Omega\times (0,T)$.
 %Wenn $C>0$ and $p>\frac{n}{2}$ mit $\norm{p}{u(\cdot,t)}\leq C$ existieren, gibt es for every $T<\Tmax$ ein $d>0$, sodass $v\geq d$ and damit insbesondere $\tel{v}\leq\tel{d}$ on  $\Omega\times [0,T]$.
\end{lemma}
\begin{proof}
 By Lemma \ref{w} there are $D=D(T)>0$ with $w\leq D$ on  $\Omega\times(0,T)$. From the definition $w:=-\log\left(\frac{v}{\norm{\infty}{v_0}}\right)$ we directly obtain $v\geq \norm{\infty}{v_0}\euler^{-D}=:d>0$ on  $\Omega\times(0,T)$.
\end{proof}

% nötig? Oder kann das weg? 
Of course, this does not prove that $\tel v$ is bounded on finite time intervals in general. After all, the existence of some $p>\f n2$ such that $\norm{p}u$ is bounded, is non-obvious. 
Thanks to Lemma \ref{u}, however, at least in the one dimensional case 
the uniform positivity of $v$ on finite time intervals can be ensured: 

\begin{kor}\label{v1d}
 If $\Omega\subseteq \R$, for every finite $T\leq \Tmax$ there is $d>0$ satisfying $v\geq d$ on  $\Omega\times(0,T)$.
\end{kor}
\begin{proof}
 Lemma \ref{u} guarantees the boundedness of $\norm{1}{u(\cdot,t)}$ on  $(0,T)$. Since $1>\frac{n}{2}=\frac{1}{2}$, the claim immediately results from \ref{vpos}.
\end{proof}
Note that this does not imply boundedness of $v$ from below if $\Tmax=\infty$. 

%For now, let us employ the basic observations of this section to obtain more insight into the behaviour of solutions to \eqref{sys}. 

\section{Global existence of classical solutions}\label{kap:global}
This section is concerned with globality of solutions. According to Theorem \ref{satz:lok}, given initial data  $u_0\in C^0(\overline{\Omega})$, $v_0\in W^{1,\infty}(\Omega)$ with $u_0\geq 0$, $v_0>0$ in $\overline{\Omega}$, there is a local solution $(u,v)$ to system \eqref{sys}. This solution is global (i.e. $\Tmax=\infty$), if for some $q>n$ for every $T\in(0,\Tmax)$, $T<\infty$, there is $C=C(T)$ such that 
\begin{equation}
 \norm{\infty}{u(\cdot,t)}+\wnorm{v(\cdot,t)}\leq C~\quad\text{for all } t\in(0,T).
\end{equation}
The existence of such a constant is what we are going to prove in this chapter. It depends on dimension $n$ and the size of the parameters $\chi$ and $\mu$. 
The approach in this chapter is related to that of \cite{winkler_singular}, where global classical solvability of \eqref{intro:sys-prod} with $f\equiv 0$ 
% \begin{align*}
%  u_t&=Δ u-\chi\kreuz\\
%  v_t&=Δ v-v+u
% \end{align*}
is shown for $0<\chi<\sqrt{\frac{2}{n}}$ in dimensions $n\geq 2$.\\
Unless otherwise specified (in some lemmata we will need further conditions on $χ$ and $μ$), we assume $μ>0$, $χ>0$, $κ\ge 0$ to be fixed and $u_0$, $v_0$ to be given as in \eqref{cond:init} and denote by $(u,v)$ the corresponding local solution and by $\Tmax$ its maximal time of existence, as provided by Theorem \ref{satz:lok}.

For the iterative procedure on which we will base the local-in-time boundedness result we use the following lemma resembling Lemma 2.4 of \cite{winkler_singular}.
\begin{lemma}\label{lm:gradv}
 Let $T\in (0,\Tmax]$, $r, q \in[1,\infty]$ and suppose 
 \begin{align*}
  \tel{2}+\frac{n}{2}\left(\tel{q}-\tel{r}\right)<1.
 \end{align*}
Then there is $C>0$ such that for all $t\in(0,T)$ we have 
 \begin{align*}
  \norm{r}{\nabla v(\cdot,t)}\leq C\left(1+\sup_{s\in (0,t)}\norm{q}{u(\cdot,s)}\right).
 \end{align*}
\end{lemma}%\newpage
\begin{proof}
 First let $q\leq r$.\\
 Due to the variation-of-constants formula, for all $t\in(0,T)$ we have 
 \begin{align*}
  \norm{r}{\nabla v(\cdot,t)}\leq \norm{r}{\nabla \euler^{tΔ} v_0}+\int_0^t\norm{r}{\nabla \euler^{(t-s)Δ}(uv)(\cdot,s)}\diff s
 \end{align*}
 and the semigroup estimates of \cite[Lemma 1.3 (iii) and (ii)]{lplq} entail the existence of $c_1>0$ and $λ>0$ such that 
 \begin{align*}
  \norm{\infty}{\nabla \euler^{tΔ} v_0}&\leq c_1\norm{r}{\nabla v_0} \quad\text{and}\\
  \norm{r}{\nabla \euler^{(t-s)Δ}(uv)}&\leq c_1\left(1+(t-s)^{-\tel{2}-\frac{n}{2}\left(\tel{q}-\tel{r}\right)}\right)e^{-λ(t-s)}\norm{q}{uv}\\
  &\leq c_1\left(1+(t-s)^{-\tel{2}-\frac{n}{2}\left(\tel{q}-\tel{r}\right)}\right)e^{-λ(t-s)}\norm{q}{u}\norm{\infty}{v}\\&
  \leq c_1\left(1+(t-s)^{-\tel{2}-\frac{n}{2}\left(\tel{q}-\tel{r}\right)}\right)e^{-λ(t-s)}\norm{q}{u}\norm{\infty}{v_0}
 \end{align*}
 hold for all $t\in(0,T)$, $s\leq t$, where in the last step we have employed Lemma \ref{v}. Together this results in 
 \begin{align*}
   \norm{r}{\nabla v(\cdot,t)}&\leq c_1\norm{\infty}{\nabla v_0} \\
   &+c_1\norm{\infty}{v_0}\sup_{s\in(0,t)}\norm{q}{u(\cdot,s)}\int_0^t\left(1+(t-s)^{-\tel{2}-\frac{n}{2}\left(\tel{q}-\tel{r}\right)}\right)e^{-λ(t-s)}\diff s
 \end{align*}
 for all $t\in(0,T)$ and hence in the claim, because the integral $\int_0^{∞} \kl{1+s^{-\tel{2}-\frac{n}{2}\left(\tel{q}-\tel{r}\right)}}e^{-λs}\diff s$ is finite. \\
For $q>r$ the claim follows from the previous considerations together with Hölder's inequality applied with the exponent $\theta=\frac{q}{r}$: For some $c_2>0$ we have 
 \begin{align*}
  \norm{r}{\nabla v(\cdot,t)}&\leq c_2\left(1+\sup_{s\in (0,t)}\norm{r}{u(\cdot,s)}\right) \leq c_2\left(1+\sup_{s\in (0,t)}\norm{q}{u(\cdot,s)}|\Omega|^{\frac{q-r}{rq}}\right)
 \end{align*}
 for all $t\in(0,T)$.
\end{proof}

Repeated application of this lemma and semigroup estimates in the first equation ensures that boundedness of $\norm{p}{u(\cdot,t)}$ for some $p>\f n2$ is sufficient to guarantee boundedness (and hence, extensibility) of the solution. 

\begin{lemma}\label{lm:bd:u:nhalbe}
 Let $T\in (0,\Tmax]$, $T<\infty$, and suppose that with some $p\ge 1$ satisfying $p>\f n2$, 
\[
 \sup_{t\in(0,T)}\norm{p}{u(\cdot,t)}<\infty. 
\]
Then 
\[
 \sup_{t\in(0,T)}\kl{\norm{∞}{u(\cdot,t)}+\Norm{W^{1,\infty}(\Om)}{v(\cdot,t)}}<\infty.
\]
\end{lemma}
\begin{proof}
Let $φ(x):=\f x2+\f14 \f{nx}{n-x}=\f{x(3n-2x)}{4(n-x)}$ if $x<n$, $φ(x)=\infty$ if $x>n$. Then for positive $x$ one has $φ(x)=x$ if and only if $x=\f n2$, and $\f{d}{dx} [φ(x)-x]=\f14 \f{n(n-x)+nx}{(n-x)^2}-\f12=-\f{(x-\f{2+\sqrt2}{2}n)(x-\f{2-\sqrt2}2 n)}{2(n-x)^2}$, $x\in (0,n)$, which is positive whenever $x\in(\f n2,n)⊂(\f{2-\sqrt2}2n,n)$, so that for these $x$,
\begin{equation}\label{pkpeleq}
 x<φ(x)<φ(x)+(φ(x)-x) = \f12 \f1{\f1x-\f1n}.
\end{equation}
We pick some $p_0>\f n2$ such that $\sup_{s\in(0,T)}\norm{p_0}{u(\cdot,s)}<\infty$ and recursively define $p_{k+1}:=φ(p_k)$ for $k\inℕ_0$. [Should, for some $k_0\inℕ_0$, $p_{k_0}=n$, we instead let $p_{k_0+1}:=\f34n \in (\f{7-\sqrt{17}}4n,p_{k_0})$, which due to monotonicity of $φ$ and $φ(\f{7-\sqrt{17}}4n)=n$ ensures $p_{k_0+2}>n$ and hence $p_{k_0+3}=∞$.] 
From monotonicity of $x\mapsto φ(x)-x$ on $(p_0,n)$ we can conclude that $p_k=\infty$ for some finite $k$. We proceed to show that 
\begin{equation}\label{impl}
 \forall k\in ℕ_0:\quad  \kl{\sup_{s\in(0,T)} \norm{p_k}{u(\cdot,s)}<∞ \implies \sup_{t\in(0,T)} \norm{p_{k+1}}{u(\cdot,t)}<∞}.
\end{equation}
By the definition of $p_{k+1}$, we have that (either the exceptional case $p_k=n$ and thus $p_{k+1}<p_k$ has occured, for which \eqref{impl} is trivial and which shall hence be ignored in the following or that) $\f1{p_{k+1}}>\f2{p_k}-\f2n$ by \eqref{pkpeleq} if $p_k<n$ (and with $\f1{p_{k+1}}=\f1{∞}:=0$ if $p_k>n$). Therefore, apparently, 
\[
 -\f12-\f n2\kl{\f1{p_{k}}+\f1r-\f1{p_{k+1}}} > -\f12 -\f n2\kl{\f1{p_k}+\f1r-\f2{p_k}+\f2n} = 1, 
\]
if $\f1r=\f1{p_k}-\f1n$, and it is hence possible to choose $r\in[1,∞]$ such that 
\begin{align}
 \f1r &>\f1{p_k} - \f1n\label{rleq}\\
\text{and }\quad -\f12-\f n2\kl{\f1{p_k}+\f1r-\f1{p_{k+1}}} & > -1 \label{rgeq}.
\end{align}
From the supposed bound on $\norm{p_k}{u(\cdot,s)}$, $s\in(0,T)$, via Lemma \ref{lm:gradv} and due to \eqref{rgeq} we obtain that also $\sup_{s\in(0,T)}\norm{r}{\nabla v(\cdot,s)}$ is finite and hence by Hölder's inequality so is $\sup_{s\in(0,T)}\Norm{\f1{\f1{p_k}+\f1r}}{u(\cdot,s)\na v(\cdot,s)}\!\!$. %$\sup_{s\in(0,T)}\norm{\f1{\f1{p_k}+\f1r}}{u(\cdot,s)\na v(\cdot,s)}\!\!$. 
With $c_1:=\sup_{ξ>0} f(ξ)$ we have 
\begin{align*}
 0 \leq u(\cdot,t) &= e^{tΔ} u_0 + \int_0^t e^{(t-s)Δ} \na \cdot\kl{\f uv \na v (\cdot,s)} ds + \int_0^t e^{(t-s)Δ} f(u(\cdot,s)) ds\\
 &\le \norm{\infty}{u_0} + c_1 T + \int_0^t e^{(t-s)Δ}\na \cdot \kl{\f uv \na v (\cdot,s)} ds \qquad \text{in } \Om \text{ for any } t\in(0,T).
\end{align*}
In order to estimate $\sup_{t\in(0,T)}\Norm{p_{k+1}}{u(\cdot,t)}$, it hence suffices to control $\Norm{p_{k+1}}{\int_0^te^{(t-s)Δ}\na\cdot\kl{\f uv\na v(\cdot,s)}ds}$, which we do by means of semigroup estimates (\cite[Lemma 1.3 (iv)]{lplq}), noting that with $c_2>0$ being the constant given there we have 
\begin{align}\label{hereweneedthelowerboundforv}
 \int_0^t &\norm{p_{k+1}}{e^{(t-s)Δ}\na \cdot\kl{\f uv\na v(\cdot,s)}}ds \nn\\
&\le c_2c_3 \int_0^{t}\kl{1+(t-s)^{-\f12-\f n2(\f1{p_k}+\f1r-\f1{p_{k+1}})}} \norm{\f1{\f1{p_k}+\f1r}}{u(\cdot,s)\na v(\cdot,s)}ds, 
\end{align}
where $c_3>0$ is a constant such that $\f1v<c_3$ in $\Om\times(0,T)$, which exists due to $p_k>\f n2$ and Lemma \ref{vpos}. This concludes the proof of \eqref{impl} due to \eqref{rgeq}. As observed before, $p_k=\infty$ for some finite $k$; the boundedness assertion concerning $\Norm{W^{1,\infty}(\Om)}{v}$ thus results from Lemma \ref{vinfty} and Lemma \ref{lm:gradv}. 
% By Lemma \ref{vpos}, $\f1v\le c_1$.
% 
% Since $p>\f n2$, there is $r>n$ such that $\f1p - \f1n <\f1r$ and hence $\f12+\f n2(\f1p-\f1r)<1$ so that by Lemma \ref{lm:gradv}, 
% \[
%  \sup_{t\in(0,T)} \io |\na v(\cdot,t)|^r < c_2.
% \]
% 
% With $c_3:=|\Om|\sup_{ξ>0} f(ξ)$, we obtain 
% 
% \[
%  \f1p \f{d}{dt} \io u^p \le -(p-1)\io u^{p-2} |\na u|^2 + c_1 χ\io u^{p-1} |\na u||\na v| + c_3
% \]
% in $(0,T)$, where two applications of Young's inequality make us end up with 
% \begin{align*}
%  \f1p \f{d}{dt} \io u^p &\le -\f{(p-1)}2 \io u^{p-2}|\na u|^2 +\f{c_1^2χ^2}{2} u^p|\na v|^2 + c_3\\
%  &\le -\f{2(p-1)}{p^2} \io |\na u^{\f p2}|^2 + ε \io u^{\f{pq}{q-2}} + C(ε)\io |\na v|^q + c_3 %q --> r
% \end{align*}
\end{proof}

% \begin{bemerkung}\label{bem:grad}
%  In particular, for $q>n$ we have: For every $r\in [1,\infty]$ there is $C(T)>0$ such that for $t\in(0,T)$
%  \begin{align*}
%   \norm{r}{\nabla v(\cdot,t)}\leq C\left(1+\sup_{s\in (0,t)}\norm{q}{u(\cdot,s)}\right).
%  \end{align*}
% \end{bemerkung}
% \begin{proof}
% If $q>n$, for every $r>0$
%  \begin{align*}
%   \tel{2}+\frac{n}{2}\left(\tel{q}-\tel{r}\right)=\tel{2}+\frac{n}{2q}-\frac{n}{2r}<\tel{2}+\frac{q}{2q}-\frac{n}{2r}<1-\frac{n}{2r}<1.
%  \end{align*}
% An application of Lemma \ref{lm:gradv} yields the desired conclusion.
% \end{proof}

We are still in need of an estimate of $\norm{p}{u(\cdot,t)}$ for some $p>\f n2$ as starting point for the above iterative procedure. This estimate will be based on the following observation.

\begin{lemma}\label{dtupvq} %Let $α=2$. 
 For all $p,q\in\R$, on  $(0,\Tmax)$ we have 
 \begin{equation*}\begin{split}
  \dt \io u^pv^q&\le -p(p-1)\io u^{p-2}v^q |\nabla u|^2+(p(p-1)\chi-2pq)\io u^{p-1}v^{q-1}\nabla u\cdot \nabla v \\
  &+(pq\chi-q(q-1))\io u^pv^{q-2}|\nabla v|^2+p\kappa\io u^pv^q-(\mu p+q)\io u^{p+1}v^q.
 \end{split}\end{equation*}
\end{lemma}
\begin{proof}
A straightforward calculation resting on integration by parts and the Neumann boundary conditions yields
 \begin{align*}
  \dt \io u^pv^q=&\io pu^{p-1}u_tv^q+\io qu^pv^{q-1}v_t\\
  =&p\io u^{p-1}v^qΔ u - p\chi\io u^{p-1}v^q\kreuz+p \io u^{p-1}f(u)v^q\\&+q\io u^pv^{q-1}Δ v-q\io u^{p+1}v^q\\
  \le &-p\io \nabla(u^{p-1}v^q)\cdot\nabla u + p\chi\io  \frac{u}{v}\nabla(u^{p-1}v^q)\cdot\nabla v + p\kappa\io u^p v^q \\&-\mu p \io u^{p+1}v^q-q\io\nabla(u^pv^{q-1})\cdot\nabla v - q\io u^{p+1}{v^q}\\
  =&-p(p-1)\io u^{p-2}v^q|\nabla u|^2-pq\io u^{p-1}v^{q-1}\nabla u \cdot \nabla v \\&+ p(p-1)\chi\io u^{p-1}v^{q-1}\nabla u \cdot \nabla v + pq\chi \io u^pv^{q-2}|\nabla v|^2+p\kappa\io u^p v^q \\&- (\mu p +q)\io u^{p+1}v^q - pq\io u^{p-1}v^{q-1}\nabla u \cdot \nabla v -q(q-1)\io u^p v^{q-2}|\nabla v|^2\\
  =&-p(p-1)\io u^{p-2}v^q|\nabla u|^2+(p(p-1)\chi-2pq)\io u^{p-1}v^{q-1}\nabla u \cdot \nabla v\\ &+(pq\chi-q(q-1))\io u^pv^{q-2}|\nabla v|^2+p\kappa\io u^p v^q- (\mu p +q)\io u^{p+1}v^q
 \end{align*}
 on  $(0,\Tmax)$.
\end{proof}

Next, we transform this differential inequality into a bound on $\io u^pv^q$, where we will, in fact, use a negative exponent $q$. 

\begin{lemma}\label{upvr} %Let $α=2$. 
 If $p>1$ and $r>0$ satisfy $p\in \left(1,\tel{\chi^2}\right)$ and $r\in (r_-,\min\{r_+,\mu p\})$, where  
 \begin{align*}
  r_{\pm}=\frac{p-1}{2}\left(1\pm\sqrt{1-p\chi^2}\right),
 \end{align*}
 and $T\in(0,\Tmax]$, $T<\infty$, then there is $C=C(T)>0$ such that 
 \begin{align*}
  \io u^pv^{-r}\leq C \quad\text{on  }(0,T).
 \end{align*}
\end{lemma}
\begin{proof}
Inserting $q=-r$ in Lemma \ref{dtupvq}, on $(0,\Tmax)$ we obtain 
 \begin{align}
  \dt \io u^pv^{-r} \le &-p(p-1)\io u^{p-2}v^{-r} |\nabla u|^2\nn \\
  &+(p(p-1)\chi+2pr)\io u^{p-1}v^{-r-1}\nabla u \cdot \nabla v \nn \\
  &-(pr\chi+r(r+1))\io u^pv^{-r-2}|\nabla v|^2\nn\\
  &+p\kappa\io u^pv^{-r}+(r-\mu p)\io u^{p+1}v^{-r}.\label{hieristalphawichtig}
 \end{align}
By Young's inequality, the second term can be estimated by  
%  Wählt man nun in der Young'schen Ungleichung (\ref{youngspez}) ${a=u^{\frac{p-2}{2}}v^{-\frac{r}{2}}|\nabla u|}$, \\${b=u^{\frac{p}{2}}v^{-\frac{r-2}{2}}|\nabla v|}$, ${\beta=p(p-1)}$ and ${\eta=(p(p-1)\chi+2pr)}$, so lässt sich der gemischte Term abschätzen:
 \begin{align*}
  &\left\lvert (p(p-1)\chi+2pr)\io u^{p-1}v^{-r-1}\nabla u \cdot \nabla v \right\rvert \\
%  &\leq (p(p-1)\chi+2pr)\io u^{p-1}v^{-r-1}|\nabla u||\nabla v|\\
  &\leq p(p-1)\io u^{p-2}v^{-r}|\nabla u|^2+\frac{(p(p-1)\chi+2pr)^2}{4p(p-1)}\io u^pv^{-r-2}|\nabla v|^2
 \end{align*}
 on  $(0,\Tmax)$. Thus, on $(0,\Tmax)$ we have 
 \begin{align*}
  \dt \io u^pv^{-r}\leq& \left(\frac{(p(p-1)\chi+2pr)^2}{4p(p-1)}-(pr\chi+r(r+1))\right)\io u^pv^{-r-2}|\nabla v|^2\\&+p\kappa\io u^pv^{-r}+(r-\mu p)\io u^{p+1}v^{-r}.
 \end{align*}
By choice of $r$, $r-\mu p<0$ and  $\frac{(p(p-1)\chi+2pr)^2}{4p(p-1)}-(pr\chi+r(r+1))<0$, because 
 \begin{align*}
  r\in(r_-,r_+) &\Rightarrow r^2-(p-1)r+\frac{p(p-1)^2\chi^2}{4}<0 \\
  &\Rightarrow p(p-1)\chi^2+4pr\chi +\frac{4r^2p}{p-1}<4pr\chi+4r^2+4r\\
  &\Rightarrow \frac{(\chi+\frac{2r}{p-1})^2p(p-1)}{4(pr\chi+r(r+1))}<1\\
  &\Rightarrow \frac{(p(p-1)\chi+2pr)^2}{4(pr\chi+r(r+1)}<p(p-1)\\
  &\Rightarrow \frac{(p(p-1)\chi+2pr)^2}{4p(p-1)}-(pr\chi+r(r+1))<0.
 \end{align*}
 We may conclude that 
 \begin{align*}
  \dt \io u^pv^{-r}\leq p\kappa\io u^pv^{-r}
 \end{align*}
 on  $(0,\Tmax)$ and hence 
 \begin{align*}
  \io u^p(\cdot,t)v^{-r}(\cdot,t) \leq \euler^{p\kappa t}\io u_0^pv_0^{-r}\quad\text{for every }t\in(0,\Tmax).
 \end{align*}
 Thus for every $T\in(0,\Tmax]$ with $T<\infty$ there is $C:=\euler^{p\kappa T}\io u_0^pv_0^{-r}$, such that 
 \begin{equation*}
  \io u^p(\cdot,t)v^{-r}(\cdot,t) \leq C \quad\text{for all }~t\in(0,T). \qedhere
 \end{equation*}
 \end{proof}

\begin{bemerkung}
The choice of $α=2$ in $f(u)\le κu-μu^{α}$ (cf. condition \eqref{cond:f} on $f$) can be seen to be important in \eqref{hieristalphawichtig}, where the last terms, which for general $α$ would be $r\io u^{p+1}v^{-r} - μp \io u^{p+α-1} v^{-r}$, only cancel in this case. The case of higher exponents $α$ is already covered by Remark \ref{rem:f}.
\end{bemerkung}

Aided by the previous lemma, we now can find a bound for $\norm{p}{u(\cdot,t)}$:
\begin{lemma}\label{up}
 L%et $α=2$ and l
et $p\in\kl{1,\tel{χ^2}}$ be such that $μp>\f{p-1}2$ and let $T\in(0,\Tmax], T<\infty$. Then there is ${C=C(T)>0}$ satisfying $\norm{p}{u(\cdot,t)}\leq C$ for all  $t\in(0,T)$.
\end{lemma}
\begin{proof}
 Let $r_\pm$ be as in the previous lemma, that is
 \begin{align*}
  r_{\pm}=\frac{p-1}{2}\left(1\pm\sqrt{1-p\chi^2}\right).
 \end{align*}
 Due to $p<\tel{\chi^2}$, apparently we have $1-p\chi^2>0$ and thus $r_-<r_+$.
 Since $μp > \f{p-1}2$, it is, moreover, ensured that $r_-<μp$, because $r_-<\f{p-1}2$.
% 
%  Since $\mu>\tel{2}$, it is, moreover, ensured that $r_-<\mu p$, because 
%  \begin{align*}
%   \sqrt{1-p\chi^2}>0 \quad\text{and}\quad\frac{(1-2\mu)p-1}{p-1}<0,\\
%   \text{hence }\quad \frac{(1-2\mu)p-1}{p-1}<\sqrt{1-p\chi^2}\\
%   \Rightarrow \frac{p-1}{2}\left(1-\sqrt{1-p\chi^2}\right)<\mu p.
%  \end{align*}
Accordingly, there is some $r\in \left(r_{-},\min\{r_+,\mu p\}\right)$. For such a number $r$ by Lemma \ref{upvr} there is $c_1>0$ satisfying 
 \begin{align*}
  \io u^p(\cdot,t)v^{-r}(\cdot,t)\leq c_1 \quad\text{for all  }t\in(0,T).
 \end{align*}
 For  $t\in (0,T)$ it now holds true that
 \begin{align*}
  \norm{p}{u(\cdot,t)}%&=\left(\io |u(\cdot,t)|^p\right)^{\tel{p}}
&=\left(\io u^p(\cdot,t) v^{-r}(\cdot,t)v^r(\cdot,t)\right)^{\tel{p}}\leq \left(\io u^p(\cdot,t) v^{-r}(\cdot,t) \norm{\infty}{v^r(\cdot,t)}\right)^{\tel{p}}\\
  &\leq \norm{\infty}{v(\cdot,t)}^\frac{r}{p}\left(\io u^p(\cdot,t) v^{-r}(\cdot,t) \right)^{\tel{p}}\leq \norm{\infty}{v_0}^\frac{r}{p} c_1^\tel{p}=:C,
 \end{align*}
 because by Lemma \ref{vinfty}, for every $t\in(0,\Tmax)$ we have $\norm{\infty}{v(\cdot,t)}\leq \norm{\infty}{v_0}$. 
\end{proof}

We can now use this to show global existence.
\begin{lemma}\label{lm:dasistthm1}
% Let $α=2$. 
For $χ<\sqrt{\f2n}$ and $μ>\f{n-2}{2n}$, system \eqref{sys} has a global solution.  
% For $\chi<\sqrt{\tel{n}}$ and $\mu>\tel{2}$ system (\ref{sys}) has a global solution.
\end{lemma}
\begin{proof}
 Because $χ<\sqrt{\f2n}$, the interval $\kl{\f n2,\f1{χ^2}}$ is nonempty. Since moreover $\f{\f n2-1}{2\cdot \f n2}=\f{n-2}{2n}<μ$, it is possible to find $p\in (\f n2,\f1{χ^2})$ such that $\f{p-1}{2p}<μ$, i.e. $μp>\f{p-1}2$. 
% Because $\chi<\sqrt{\tel{n}}$, there is $p\in \left(n,\tel{\chi^2}\right)$. According to Theorem \ref{up}, for such $p$ for 
By Lemma \ref{up} for every such $p$ and every $T\in (0,\Tmax], T<\infty$ there is $C(T)>0$ with 
 \begin{align*}
  \norm{p}{u(\cdot,t)}\leq C(T) \quad\text{for }t\in(0,T).
 \end{align*}
If we suppose that $\Tmax$ were finite, we could, herein, choose $T=\Tmax$ and from Lemma \ref{lm:bd:u:nhalbe} infer that 
\[
 \sup_{t\in(0,\Tmax)}\kl{\norm{∞}{u(\cdot,t)}+\Norm{W^{1,\infty}(\Om)}{v(\cdot,t)}}<\infty, 
\]
in blatant contradiction to \eqref{extensibilitycrit}.
\end{proof}

\begin{proof}[Proof of Theorem \ref{thm:main:ge}]
 In fact, Theorem \ref{thm:main:ge} is identical to Lemma \ref{lm:dasistthm1}.
\end{proof}

\begin{bemerkung}
Note that this does not yet show that the solution is bounded, because the positive lower bound for $v$, crucial in the estimate  \eqref{hereweneedthelowerboundforv}, was not achieved independently of time, and even \textit{cannot} be obtained in a time-independent fashion, as the example of $u_0\equiv \f{κ}{μ}$, $v_0\equiv 1$ shows. 
\end{bemerkung}

Restricting the problem to the one-dimensional setting, we will deal with boundedness in the next section.

\section{The one-dimensional case}\label{kap:1d}

If we consider \eqref{sys} in a one-dimensional domain, i.e. if $\Om\subseteqℝ$ is an interval, we can prove stronger claims. Then, namely, the solution is not only global, independently of the positive parameters and initial data $u_0\in C^0(\overline{\Omega})$, $v_0\in W^{1,\infty}(\Omega)$ with $u_0\geq0$, $v_0>0$ in $\overline{\Omega}$, but even bounded. The system in this setting is the following: 
\begin{align}\label{sys1dv}
%\begin{array}{rlll}
 u_t&=u_{xx}-\chi(\frac{u}{v} v_x)_x+ f(u) &&\text{in } \Om\times(0,\Tmax),\nn\\%&x\in\Omega, ~&t\in(0,\Tmax),\\
 v_t&=v_{xx}-uv&&\text{in } \Om\times(0,\Tmax),\\% &x\in\Omega, ~&t\in(0,\Tmax),\\
 u_x&=0=v_x &&\text{in } \partial\Om\times(0,\Tmax),\nn\\% &x\in\partial\Omega, ~&t\in(0,\Tmax),\\
 u(\cdot,0)&=u_0,\quad v(\cdot,0)=v_0&&\text{in } \Om.\nn% &x\in\Omega.&
% \end{array}
\end{align}
The goal of this section is to prove Theorem \ref{thm:main:bd1d}. 
% \begin{lemma}\label{1dbeschr}
%  For arbitrary $\chi>0, \mu>0$ and $\kappa>0$, the solution $(u,v)$ of (\ref{sys1dv}) is global and bounded.
% \end{lemma}
% \begin{bemerkung}
%  Die Globalität der Lösung folgt wegen Satz \ref{satz:lok} direkt aus der Beschränktheit.
% \end{bemerkung}
As in Section \ref{sec:vpos}, we will use $w:=-\log\left(\frac{v}{\norm{\infty}{v_0}}\right)$ for the proof, and hence have to deal with 
\begin{align}\label{sys1d}
%\begin{array}{rlll}
 u_t&=u_{xx}+\chi(u w_x)_x+ f(u) &&\text{in } \Om\times(0,\Tmax),\nn\\% &x\in\Omega, ~&t\in(0,\Tmax),\\
 w_t&=w_{xx}-w_x^2+u  &&\text{in } \Om\times(0,\Tmax),\\%&x\in\Omega, ~&t\in(0,\Tmax),\\
 u_x&=0=w_x &&\text{in } \partial \Om\times(0,\Tmax),\nn\\% &x\in\partial\Omega, ~&t\in(0,\Tmax),\\
 u(\cdot,0)&=u_0, \quad w(\cdot,0)=w_0:=-\log\left(\frac{v_0}{\norm{\infty}{v_0}}\right) &&\text{in } \Om.\nn%\\%&x\in\Omega.&
% \end{array}
\end{align}
We will, as before, denote by $(u,w)$ the local solution to \eqref{sys1d} for given, fixed initial data $u_0$, $v_0$ as in \eqref{cond:init}, for a function $f$ satisfying \eqref{cond:f} and for parameters $κ\ge 0$, $\mu>0$, $χ>0$, on which we pose no further conditions. 

Let us prepare the proof of Theorem \ref{thm:main:bd1d}, which will essentially rely on several differential inequalities and, again, semigroup estimates, with the following few lemmata. The first of these relies on Lemma \ref{odi-spatiotemporal} to turn the spatio-temporal estimate resulting from the presence of the logistic source into various pieces of boundedness information concerning derivatives of $w$. 
\begin{lemma}\label{1dw}
 There is $C>0$ such that 
\[\norm{2}{w_x(\cdot,t)}\leq C,\quad \inttnullt \io w_{xx}^2\leq C\quad \text{ and } \inttnullt \io w_x^6\leq C\] hold for all $t\in(0,\Tmax)$.
\end{lemma}
\begin{proof}
Multiplying the second equation of (\ref{sys1d}) by $-w_{xx}$ and integrating over $\Omega$, from integration by parts we obtain 
 \begin{align*}
  \tel{2}\dt\io w_x^2+\io w_{xx}^2=-\io w_{xx}u\quad\text{on  }(0,\Tmax),
 \end{align*}
 because $-\io w_t w_{xx}=\io w_{xt} w_x=-\tel{2}\dt \io w_x^2$ and $\io w_x^2 w_{xx}=\io \left(\tel{3}w_x^3\right)_x=0$ due to the homogeneous Neumann boundary conditions. From Young's inequality we obtain 
\begin{align*}
 -\io w_{xx}u\leq\io |w_{xx}||u|\leq \tel{2}\io w_{xx}^2+\tel{2}\io u^2 \qquad \text{on } (0,\Tmax),
\end{align*}
hence 
\begin{align}\label{eq:wx}
 \dt \io w_x^2+\io w_{xx}^2\leq\io u^2.
\end{align}
Poincar\'{e}'s inequality yields $C_P>0$ such that 
\begin{align*}
 \io w_x^2\leq C_P \io w_{xx}^2,\quad\text{i.e. }\io w_{xx}^2\geq\tel{C_P}\io w_x^2 \quad\text{on  }(0,\Tmax).
\end{align*}
Therefore, on $(0,\Tmax)$ we have
\begin{align*}
 \dt \io w_x^2+\tel{C_P}\io w_x^2\leq \dt \io w_x^2+\io w_{xx}^2\leq\io u^2.
\end{align*}
Thus we have derived an ordinary differential inequality of the same form as in Lemma \ref{odi-spatiotemporal} if we set $y(t)=\io w_x^2(\cdot,t)$ and $h(t)=\io u^2(\cdot,t)$, where the condition $\inttnullt h(s)\diff s\le c_1$ is satisfied for some $c_1>0$ for all $t\in(0,\Tmax)$ according to Lemma \ref{u2}. 
Hence there is $c_2>0$ with $\io w_x^2\leq c_2$ on  $(0,\Tmax)$. 
Integration of (\ref{eq:wx}) with respect to time shows that for all $t\in(0,\Tmax)$,  
\begin{align*}
 \inttnullt\io w_{xx}^2\leq \inttnullt\io u^2-\io w_x^2(\cdot,t)+\io w_x^2(\cdot,t_0)\leq c_1+0+c_2=:c_3.
\end{align*}
By the Gagliardo--Nirenberg inequality %with $y=w_x$, $p=6$, $j=0$, $m=1$, $r=2$, $q=2$, $s=2$ and $\alpha=\tel{3}$
 there are $c_4>0$ and $c_5>0$ satisfying 
\begin{align*}
 \io w_x^6 \leq c_4 \left(\io w_{xx}^2\right)\left(\io w_x^2\right)^2+c_5\left(\io w_x^2\right)^3 \;\text{throughout } (0,\Tmax),
\end{align*}
which, together with the previously shown, results in 
\begin{align*}
 \inttnullt\io w_x^6 \leq c_2^2c_3c_4+c_2^3c_5=:c_6\quad\text{for }t\in(0,\Tmax).
\end{align*}
Setting $C:=\max\left\{\sqrt{c_2}, c_3, c_6\right\}$ gives the claim. 
\end{proof}

By similar reasoning, we can derive finiteness of $\sup_t \norm{2}{u(\cdot,t)}$, which gives much stronger information than the previously known bound for $\Norm{L^2(\Om\times(t_0,t))}{u}$. Along the way we collect some further spatio-temporal bounds for $u$ and its derivative.

\begin{lemma}\label{1du}
 There is $C>0$ such that \[\norm{2}{u(\cdot,t)}\leq C\] holds for all $t\in(0,\Tmax)$ and that, furthermore, \[\inttnullt\io u_x^2\leq C\quad \text{ and }\quad \inttnullt\io u^6\leq C\] hold for all $t\in(0,\Tmax)$.
\end{lemma}
\begin{proof}
 Multiplying (\ref{sys1d}) by $u$, integrating over $\Omega$, from integration by parts we see that 
 \begin{align*}
  \tel{2}\dt \io u^2\le -\io u_x^2-\chi\io u_xuw_x+\kappa\io u^2-\mu\io u^3\quad\text{on  }(0,\Tmax).
 \end{align*}
Young's inequality shows that 
 \begin{align*}
  \chi \io u_xuw_x\leq \tel{2}\io u_x^2+\frac{\chi^2}{2}\io u^2w_x^2\leq \tel{2}\io u_x^2+\frac{\mu}{2}\io u^3+c_1\io w_x^6
 \end{align*}
with some $c_1>0$, and 
 \begin{align*}
  \kappa\io u^2\leq \frac{\mu}{2}\io u^3+\frac{\kappa^2}{2\mu}\io u.
 \end{align*}
Thus, on  $(0,\Tmax)$ we obtain 
 \begin{align}\label{eq:u}
  \dt \io u^2+\io u_x^2\leq 2c_1\io w_x^6+m\frac{\kappa^2}{\mu}
 \end{align}
 with $m$ as in Lemma \ref{u} 
%  Um daraus eine Differentialungleichung der Art wie in Lemma \ref{odi} herzuleiten, werden die Poincar\'e-Ungleichung, die Dreiecksungleichung sowie die Young'sche Ungleichung verwendet:
%  \begin{align*}
%   C_P\io u_x^2\geq&\io |u-\bar{u}|^2\geq \io (|u|-|\bar{u}|)^2=\io u^2-2\io |u||\bar{u}|+\io \bar{u}^2\\
%   \geq & \tel{2}\io u^2-\io \bar{u}^2=\tel{2}\io u^2-\bar{u}^2|\Omega|,
%  \end{align*}
and, due to Poincar\'e's inequality, hence 
 \begin{align*}
  \dt \io u^2+\tel{2C_P}\io u^2\leq \dt \io u^2+\io u_x^2+\bar{u}^2|\Omega|\leq 2c_1\io w_x^6+m\frac{\kappa^2}{\mu}+m^2|\Omega|.
 \end{align*}
 According to Lemma \ref{odi-spatiotemporal} with $h(t)=2c_1\io w_x^6(\cdot,t)+m\frac{\kappa^2}{\mu}+m^2|\Omega|$, which by Lemma \ref{1dw} satisfies the condition $\inttnullt h(s)\diff s\leq c_2$ for all $t\in(0,\Tmax)$, there is $c_3>0$ satisfying $\io u^2(\cdot,t)\leq c_3$ for every $t\in(0,\Tmax)$. \\
Integrating inequality \eqref{eq:u}, for $t\in(0,\Tmax)$, from Lemma \ref{1dw} and with $c_4$ being the constant taken from Lemma \ref{1dw}, we obtain 
 \begin{align*}
  \inttnullt\io u_x^2&\leq 2c_1\inttnullt\io w_x^2+m\frac{\kappa^2}{\mu}-\io u^2(\cdot, t)+\io u^2(\cdot,t_0)\\&\leq 2c_1c_4+m\frac{\kappa^2}{\mu}+c_3=:c_5.
 \end{align*}
 Application of the Gagliardo--Nirenberg inequality %with $y=u$, $p=6$, $j=0$, $m=1$, $r=2$, $q=2$, $s=2$ and $\alpha=\tel{3}$ 
 together with the above provides us with a positive number $c_6>0$ such that 
 \begin{align*}
  \inttnullt\io u_x^6\leq c_6\quad\text{for }t\in(0,\Tmax).
 \end{align*}
 The definition $C:=\max\left\{\sqrt{c_3},c_5,c_6\right\}$ finally ensures the validity of the claim.
\end{proof}\vspace{0.1cm}

The next step is to ascertain control on derivatives of $w$ uniformly in time. 

\begin{lemma}
 There is $C>0$ with \[\io w_x^4(\cdot,t)\leq C\] for all $t\in(0,\Tmax)$.
\end{lemma}
\begin{proof}
We will again derive an ODI. Differentiation of $\io w_x^4$ and integration by parts lead to 
 \begin{align*}
  \dt \io w_x^4=-12\io w_x^2w_{xx}^2-\frac{8}{5}\io (w_x^5)_x+4\io w_x^3u_x=-12\io w_x^2w_{xx}^2+4\io w_x^3u_x 
 \end{align*}
 on  $(0,\Tmax)$. By Young's inequality, we have 
 \begin{align*}
  4\io w_x^3u_x\leq 2\io w_x^6+2\io u_x^2 \qquad \text{on } (0,\Tmax).
 \end{align*}
 Moreover, from Poincar\'e's inequality we may infer 
 \begin{align*}
  \io w_x^4=\io \left(w_x^2\right)^2\leq C_P \io \left((w_x^2)_x\right)^2=4C_P\io w_x^2w_{xx}^2 \qquad \text{on } (0,\Tmax).
 \end{align*}
Combining these, on  $(0,\Tmax)$ we obtain 
\begin{align*}
 \dt \io w_x^4+\frac{3}{C_P}\io w_x^4\leq \dt \io w_x^4+12\io w_x^2w_{xx}^2\leq 2\io w_x^6+2\io u_x^2.
\end{align*}
Due to the estimates for $\inttnullt\io w_x^6$ and $\inttnullt\io u_x^2$ for $t\in(0,\Tmax)$, of Lemma \ref{1dw} and Lemma \ref{1du}, respectively, taken together with Lemma \ref{odi-spatiotemporal}, we can conclude the proof.
\end{proof}
With this, it is possible to prove Theorem \ref{thm:main:bd1d}.
\begin{proof}[Proof of Theorem \ref{thm:main:bd1d}]
 By the variation-of-constants formula and \eqref{cond:f}, for every $t\in(0,\Tmax)$,
 \begin{align*}
 %\begin{array}{rl}
  u(\cdot,t)=& \euler^{tΔ}u_0+\chi\int_0^t\euler^{(t-s)Δ}(uw_x)_x(\cdot,s)\diff s\nonumber%\\&
+\int_0^t\euler^{(t-s)Δ}f(u(\cdot,s))\diff s\nonumber\\
  \leq& \euler^{tΔ}u_0+\chi\int_0^t\euler^{(t-s)Δ}(uw_x)_x(\cdot,s)\diff s+\kappa\int_0^t\euler^{(t-s)Δ}u(\cdot,s)\diff s\;\text{in } \Om 
  %\norm{\infty}{u}\leq&\norm{\infty}{\euler^{tΔ}u_0}+\chi\int_0^t\norm{\infty}{\euler^{(t-s)Δ}(uw_x)_x}\\&+\kappa\int_0^t\norm{\infty}{\euler^{(t-s)Δ}u}+\mu\int_0^t\norm{\infty}{\euler^{(t-s)Δ}u^2}
  %\end{array}
 \end{align*} 
% in $\Omega$ 
 and hence 
 \begin{align}\label{1dinfty}
 %\begin{array}{rl}
 \norm{\infty}{u(\cdot,t)}\leq&\norm{\infty}{\euler^{tΔ}u_0}+\chi\int_0^t\norm{\infty}{\euler^{(t-s)Δ}(uw_x)_x(\cdot,s)}\diff s%\nonumber\\&
+\kappa\int_0^t\norm{\infty}{\euler^{(t-s)Δ}u(\cdot,s)}\diff s.
 %\end{array}
 \end{align} 
 Now we want to use the preceding two lemmata to prove boundedness of the map $(0,\Tmax)\!\ni t\mapsto \norm{\infty}{u(\cdot,t)}$. They provide constants $c_1, c_2>0$ with $\norm{2}{u(\cdot,t)}\leq c_1$ and $\norm{4}{w_x(\cdot,t)}\leq c_2$ for all $t\in(0,\Tmax)$.
Semigroup estimates (cf. \cite[Lemma 1.3]{lplq}) imply the existence of some constant $c_3>0$ (independent of $t$), such that the terms in \eqref{1dinfty} fulfil 
\begin{align*}
 \norm{\infty}{\euler^{tΔ}u_0}&\leq c_3\norm{\infty}{u_0},\\
 \norm{\infty}{\euler^{(t-s)Δ}u}&\leq c_3\left(1+(t-s)^{-\tel{4}}\right)\norm{2}{u}\leq c_1c_3 \left(1+(t-s)^{-\tel{4}}\right),\\
 \norm{\infty}{\euler^{(t-s)Δ}(uw_x)_x}&\leq c_3\left(1+(t-s)^{-\tel{2}-\frac{3}{8}}\right)\norm{\frac{4}{3}}{uw_x}\\
 &\leq c_3\left(1+(t-s)^{-\frac{7}{8}}\right)\norm{2}{u}\norm{4}{w_x}\\
 &\leq c_1c_2c_3\left(1+(t-s)^{-\frac{7}{8}}\right)\qquad\qquad\text{for }t\in(0,\Tmax),\; s\in(0,t], 
 %\anderset{\text{Hölder mit }\theta=\frac{4}{3}}{\leq}&C(1+(t-s)^{-\frac{7}{8}})\norm{2}{u}\norm{6}{w_x}\leq\tilde{C}(1+(t-s)^{-\frac{5}{6}})
 \end{align*}
where in the penultimate step we used Hölder's inequality with exponent $\frac{3}{2}$. For $t\in (0,\min\left\{1,\Tmax\right\})$ we hence have 
 \begin{align*}
   \norm{\infty}{u(\cdot,t)}\leq&c_3\norm{\infty}{u_0}+\chi c_1c_3\int_0^t \left(1+(t-s)^{-\tel{4}}\right)\diff s + \kappa c_1c_2c_3 \int_0^t \left(1+(t-s)^{-\frac{7}{8}}\right) \diff s\\
   \leq & c_3\norm{\infty}{u_0}+\frac{7}{3}\chi c_1c_3+9=:c_4, 
 \end{align*} 
 because $\int_0^t (1+(t-s))^{-\tel4} ds = t+\frac43t^{\f34} \le \f73$ and $\int_0^t (1+(t-s)^{-\f78} ds=t+8t^{\f18}\le 9$. 
 If $\Tmax>1$, let $\tau\in\left[1,\min\left\{10,\frac{\Tmax}{2}\right\}\right)$, $t\in[0,\Tmax-\tau)$. Then 
 \begin{align*}%\label{1dinfty2}
 %\begin{array}{rl}
 \norm{\infty}{u(\cdot,t+\tau)}\leq&\norm{\infty}{\euler^{\tauΔ}u(\cdot,t)}+\chi\int_0^\tau\norm{\infty}{\euler^{(\tau-s)Δ}(uw_x)_x(\cdot,t+s)}\diff s \nn \\&+\kappa\int_0^\tau\norm{\infty}{\euler^{(\tau-s)Δ}u(\cdot,t+s)}\diff s.
 %\end{array}
 \end{align*} 
Lemma \ref{u} provides a constant $m$ such that, according to the semigroup estimates of Lemma 1.3 (i) of \cite{lplq} we have 
 \begin{align*}
  \norm{\infty}{\euler^{\tauΔ}u(\cdot,t)}\leq c_3(1+\tau^{-\tel{2}})\norm{1}{u(\cdot,t)}\leq 2c_3m.
 \end{align*}
The other terms can be estimated as above and we obtain 
 \begin{align*}
  \norm{\infty}{u(\cdot,t+\tau)}\leq& 2c_3m++\chi c_1c_3\int_0^\tau (1+(\tau-s)^{-\tel{4}})\diff s\\%\underbrace{\int_0^\tau (1+(\tau-s)^{-\tel{4}})\diff s}_{\substack{=\tau+\frac{4}{3}\tau^{\frac{3}{4}}\leq 10+\frac{4}{3}10^{\frac{3}{4}}\leq 42}}\\
   &+ \kappa c_1c_2c_3 \int_0^\tau (1+(\tau-s)^{-\frac{7}{8}}) \diff s%\underbrace{\int_0^\tau (1+(\tau-s)^{-\frac{7}{8}}) \diff s}_{=\tau+8\tau^{\frac{1}{8}}\leq 10+8 \sqrt[8]{10}\leq 42}\\
%   \leq & 2Cm+42\chi C c_1+42\kappa Cc_1c_2
=:c_5
 \end{align*}
  and therefore $\norm{\infty}{u(\cdot,t)}\leq c_5$ for all $t\in [1,\Tmax).$

  If we define $C:=\max\{c_4,c_5\}$, we accordingly obtain $\norm{\infty}{u(\cdot,t)}\leq C$ for every $t\in(0,\Tmax)$; hence $u$ is bounded. 

  Since, furthermore, by Lemma \ref{v} for $q>n$ we have 
  \begin{align*}
  \norm{q}{v(\cdot,t)}\leq\norm{q}{v_0}\quad\text{for all }t\in(0,\Tmax), 
  \end{align*}
  and, due to Lemma \ref{lm:gradv}, moreover 
  \begin{align*}
  \norm{q}{\nabla v(\cdot,t)}\leq c_6\left(1+\sup_{s\in(0,\Tmax)}\norm{q}{u(\cdot,s)}\right)
  \end{align*}
  holds for some $c_6>0$, which ensures that for any $q>n$ also the function 
  \begin{align*}
  (0,\Tmax)\ni t\mapsto \wnorm{v(\cdot,t)}
  \end{align*}
 is bounded, the extensibility property from Theorem \ref{satz:lok} shows that the solution $(u,v)$ is global.% and bounded.
\end{proof}
% 
% \section{Boundedness if $κ=0$ and $α>1+\f n2$}
% 
% \begin{lemma}
%  Let $α>1+\f1n$ and suppose that $T\in(0,\Tmax]$ is such that 
% \[
%   \int_0^T \io u^{α} < \infty. 
% \]
%  Then 
% \[
%  \sup_{(x,t)\in \Om\times (0,T)} w(x,t) <\infty, \quad \text{i.e. }\quad \sup_{(x,t)\in \Om\times (0,T)} \f1{v(x,t)} <\infty.
% \]
% \end{lemma}
% \begin{proof}
%  
% \end{proof}
% 
% 
% 
% 
% 
% \begin{bemerkung}
% Note that for the proof of (global-in-time) boundedness of $w$, on which all of this section relies, the condition $κ=0$ is essential. For positive $κ$ consider $u_0\equiv (\f{κ}{μ})^{\f1{α-1}}$, $v_0\equiv 1$, which apparently leads to $u\equiv const.$ and $v(\cdot,t)\equiv \exp{-(\f{κ}{μ})^{\f1{α-1}}t}$ for all $t>0$, and hence $w(\cdot,t)=(\f{κ}{μ})^{\f1{α-1}}t$, which is unbounded. %ABER w TRITT IN DER ERSTEN GLEICHUNG NICHT AUF; NUR SEIN GRADIENT. 
% \end{bemerkung}

%\section{Acknowledgements}

\begin{appendix}
\section{Proof of the local existence theorem (Theorem \ref{satz:lok})}\label{sec:locex-proof}
\begin{proof}
 Because $v_0\in W^{1,\infty}(\Omega)$ and $\Omega$ is bounded, we apparently have $v_0\in W^{1,q}(\Omega)$. 
 We let $S$, $f$ and $g$ be as in (\ref{susw}). Then $g\in C_{\text{loc}}^{1-}(\overline{\Omega}\times[0,\infty)\times\R^2)$ and $g(x,t,u,0)=0$. 

 For every \mbox{$k\in\N\cap(\max\left\{\norm{\infty}{u_0},\wnorm{v_0}\right\},\infty)$} we define 
 \begin{align*}
 \eta_k:=\inf_{x\in\Omega}v_0(x)\euler^{-(k+1)^2}.
 \end{align*}
Moreover, we let $ζ_k\colon \R\rightarrow [0,1]$ be a smooth, monotone decreasing function with $ζ_k(\frac{\eta_k}{2})=1$ and $ζ_k(\eta_k)=0$. Furthermore, we set 
 \begin{align*}
  S_k&\colon \overline{\Omega}\times[0,\infty)\times\R^2\rightarrow \R, \qquad& S_k(x,t,u,v)&:=\begin{cases}
                                                                                      \frac{2\chi}{\eta_k},&v\leq\frac{\eta_k}{2},\\
                                                                                      ζ_k(v)\frac{2\chi}{\eta_k}+(1-ζ_k(v))\frac{\chi}{v},&v>\frac{\eta_k}{2},
                                                                                     \end{cases}
 \end{align*}
 and
 \begin{align*}
  h_k&\colon\overline{\Omega}\times[0,\infty)\times\R^2\rightarrow \R, \qquad & h_k(x,t,u,v)&:=\begin{cases}
										     f'(0) u,&u<0,\\ 
                                                                                     f(u),&0\leq u\leq k+1,\\
                                                                                     \kappa(k+1)-\mu(k+1)^2\\\quad+f'(k+1)(u-(k+1)),&u>k+1.
                                                                                    \end{cases}
 \end{align*}
Then by definition we have $S_k\in C_{\text{loc}}^{2}(\overline{\Omega}\times[0,\infty)\times\R^2)$, $h_k(x,t,0,v)=0$ and $h_k\in C^{1-}(\overline{\Omega}\times[0,\infty)\times\R^2)$, since $h_k\in C^1(\overline{\Omega}\times[0,\infty)\times\R^2)$ with bounded derivative. 
% \\\ubar{Behauptung 1:} $f_k$ ist Lipschitz-stetig bezüglich $u$.\\
% \textit{Beweis von Behauptung 1:} Es sind drei Fälle zu unterscheiden: $u,\hat{u}$ 
Hence  $S_k$, $h_k$ and $g$ satisfy the conditions of Lemma  \ref{satz:survey}. 
Accordingly, there are $T_{\max,k}\in(0,\infty]$ and a unique pair of functions $(u_k,v_k)$ with $u_k \in C^0(\overline{\Omega}\times[0,T_{\max,k}))\cap C^{2,1}(\overline{\Omega}\times(0,T_{\max,k}))$ and $v_k\in C^0(\overline{\Omega}\times[0,T_{\max,k}))\cap C^{2,1}(\overline{\Omega}\times(0,T_{\max,k}))\cap L_{\text{loc}}^{\infty}([0,T_{\max,k});W^{1,q}(\Omega))$, such that $(u_k,v_k)$ is a classical solution to 
\begin{align*}
  \left.\begin{array}{lr}
   u_t=Δ u -\nabla\cdot\left(uS_k(x,t,u,v)\nabla v\right)+h_k(x,t,u,v),&x\in\Omega, ~t>0,\\
   v_t=Δ v-v+g(x,t,u,v),&x\in\Omega, ~t>0,\\
   \frac{\partial u}{\partial \nu}=\frac{\partial v}{\partial \nu}=0,&x\in\partial\Omega, ~t>0,\\
   u(x,0)=u_0(x), v(x,0)=v_0(x),&x\in\Omega,
  \end{array}\right\}(P_k)
 \end{align*}
and that either $T_{\max,k}=\infty$ or $\norm{\infty}{u_k(\cdot,t)}+\wnorm{v_k(\cdot,t)}\rightarrow \infty$ for ${t\nearrow T_{\max,k}}$.\\
Comparison with the lower solution $\ubar{u}=0$ due to the nonnegativity of $u_0$ by the comparison theorem \ref{vs} (note also Remark \ref{bem:vs} concerning its applicability) immediately yields $u_k\geq 0$, so that the definition of $h_k$ for $u<0$ plays no further role. 
We now let 
\begin{align*}
t_k:=\min\left\{k,\inf\{t\in(0,T_{\max,k})\big\vert\sup_{x\in\Omega}u_k(x,t)>k \quad\text{or } \wnorm{v_k(\cdot,t)}>k\}\right\}.
 \end{align*}
%Dann gilt nach Definition von $t_k$ offenbar $t_k\leq T_{\max,k}$ and $u_k\leq k$ on  $\Omega\times[0,t_k]$.\\\\
 %%%ALTE VERSION
 Then $t_k< T_{\max,k}$, for if $T_{\max,k}\leq k$, then by 
\begin{align*}
\norm{\infty}{u_k(\cdot,t)}+\wnorm{v_k(\cdot,t)}\rightarrow \infty\quad\text{for }t\nearrow T_{\max,k}
\end{align*}
it holds that 
\begin{align*}
 \inf&\left\{t\in(0,T_{\max,k})\big\vert\sup_{x\in\Omega}u_k(x,t)>k \quad\text{oder } \wnorm{v_k(\cdot,t)}>k\right\}\\
 &\leq  \inf\left\{t\in(0,T_{\max,k})\vert\norm{\infty}{u_k(\cdot,t)}+\wnorm{v_k(\cdot,t)}>k\right\}< T_{\max,k}.
\end{align*}
By definition of $t_k$, on $\Omega\times[0,t_k]$, we have $u_k\leq k$. \\
\underline{Claim 1:} On $\Omega\times[0,t_k]$ we have $v_k\geq \eta_{k-1}$.\\
\textit{Proof of claim 1:} If we let $\uvk(x,t):=\inf_{y\in\Omega}v_0(y)\euler^{-kt}$, due to $u_k\leq k$ on  $\Omega\times[0,t_k]$ and $Δ \uvk=0$ we have:
\begin{align*}
 \ubar{v}_{k_t}&=-k\uvk\leq Δ \uvk-u_k \uvk&&\text{in }\Omega\times(0,t_k],\\
 \partial_{\nu}\uvk&=0\quad&&\text{on  }\partial\Omega,\\
 \uvk(x,0)&=\inf_{y\in\Omega}v_0(y)\leq v_0(x)\quad&&\text{for all } x\in\Omega.
\end{align*}
Thus by the comparison theorem \ref{vs}, for all  $x\in \Omega$, $t\leq t_k\colon$
\begin{align*}
 v_k(x,t)\geq \uvk(x,t)=\inf_{y\in\Omega}v_0(y)\euler^{-kt}\geq \inf_{y\in\Omega}v_0(y)\euler^{-k^2}=\eta_{k-1}.
\end{align*}
\underline{Claim 2:} On $\Omega\times(0,t_k)$ the function $(u_k,v_k)$ coincides with $(u_{k+1},v_{k+1})$.\\
\textit{Proof of Claim 2:} On $\Omega\times(0,t_{k+1})$ it holds that $v_{k+1}\geq \eta_k\geq\eta_{k+1}$, therefore 
\begin{align*} 
S_{k+1}(x,t,u_{k+1},v_{k+1})=S_k(x,t,u_{k+1},v_{k+1})=S(x,t,u_{k+1},v_{k+1}). 
\end{align*} 
Moreover $u_{k+1}\leq k+1$ and thus 
\begin{align*}
h_{k+1}(x,t,u_{k+1},v_{k+1})=h_k(x,t,u_{k+1},v_{k+1})=h(x,t,u_{k+1},v_{k+1}).
\end{align*}
Consequently, $(u_{k+1},v_{k+1})$ solves $(P_k)$ on  $\Omega\times(0,\min\{t_k,t_{k+1}\})$ and therefore, by uniqueness of the solution, has to satisfy $(u_k,v_k)=(u_{k+1},v_{k+1})$ on  $\Omega\times(0,\min\{t_k,t_{k+1}\})$.\\
For contradiction, we now assume that $t_k>t_{k+1}$. Since $t_k\leq k$, then  $t_{k+1}\neq k+1$ would have to hold and thus $\sup_x u_{k+1}(x,t_{k+1})=k+1$ or $\wnorm{ v_{k+1}(\cdot, t_{k+1})}=k+1$. But this would entail the existence of $\tilde{t}<t_{k+1}$ with $\sup_x u_{k+1}(x,\tilde{t})=k$ or  $\wnorm{ v_{k+1}(\cdot, \tilde{t})}=k$, which would imply $t_k\leq \tilde{t}<t_{k+1}$ in contradiction to the assumption.\\
Hence, $t_k\leq t_{k+1}$ and $(u_k,v_k)=(u_{k+1},v_{k+1})$ on  $\Omega\times(0,t_k)$.\\
Because $(t_k)_k$ is a monotone increasing sequence, the limit $\Tmax:=\lim_{k\to \infty} t_k\in (0,\infty]$ exists. 
Therefore, the functions $u$ and $v$ given by 
\begin{align*}
 u(x,t)=u_k(x,t), &\quad x\in\Om, \text{  for any $k$ such that }t<t_k,\; \\
 v(x,t)=v_k(x,t), &\quad x\in\Om, \text{  for any $k$ such that }t<t_k,\;
\end{align*}
are well-defined on $\Omega\times [0,\Tmax)$. On  $\Omega\times(0,t_k)$ we already know  $S_k=S$ and $h_k=h$,  thus $(u,v)$ is indeed a classical solution of \eqref{sys} on  $\Omega\times(0,\Tmax)$, because the regularity of $u_k$ and $v_k$ directly implies that of $u$ and $v$. 
If, moreover, $\Tmax<\infty$, for all $k>\Tmax$:
\begin{align*}
 \limsup_{t\nearrow \Tmax}\left(\norm{\infty}{u(\cdot,t)}+\wnorm{ v(\cdot,t)}\right)\geq k
\end{align*}
and hence $\limsup_{t\nearrow \Tmax}\kl{\norm{\infty}{u(\cdot,t)}+\wnorm{v(\cdot,t)}}=\infty$. 
If there were a second pair $(\tilde{u},\tilde{v})$ of functions solving \eqref{sys} on $\Omega\times (0,\Tmax)$, by uniqueness of solutions to $(P_k)$, this would have to coincide with $(u_k,v_k)$ on $(0,t_k)$, which proves uniqueness of solutions.
\end{proof}
% \begin{satz}
%  Wenn $u_0\in$ bla and $v_0\in$ bla, dann existiert $\Tmax\in (0,\infty]$ and ein eindeutiges Paar $(u,v)$, das das System \ref{sys} in $\Omega\times [0,\Tmax]$ löst. Außerdem gilt $\Tmax=\infty$ oder
%  \begin{equation}
%    \limsup_{t\nearrow \Tmax} \Vert u(\cdot,t) \Vert_{L^\infty(\Omega)} = \infty.
%  \end{equation}
% \end{satz}

\section{A comparison theorem}\label{sec:comparison}
 An extremely useful tool for the proof not only of the existence assertion, but also the most basic solution properties of Section \ref{kap:eigenschaften} is the following comparison theorem. Comparison theorems of this form are often used, but seldom proven or referenced -- in fact, we did not find a version applicable to (the first equation of) the present system with Neumann boundary conditions in the literature; \cite[Prop. 52.7]{quittner_souplet}, which seemed the closest match, additionally requires $\nabla \bar u, \nabla \ubar u\in L^{∞}(\Om\times(0,T))$ --, for the sake of completeness, we therefore have decided to include a suitable version and its proof: 

\begin{satz}[Comparison theorem]\label{vs}
 Let $\Omega\subseteq \R^n$ be a bounded domain with smooth boundary, $T>0$ and let $f\in C^0(\overline{\Omega}\times[0,T)\times\R)$ satisfy a local Lipschitz condition in the following sense: For every compact $K\subseteq \R$ there be $L(K)>0$ such that 
 \begin{align*}
  |f(x,t,u)-f(x,t,v)|\leq L(K)|u-v| \quad\text{for all } x\in\overline{\Omega}, t\in [0,T), u,v\in K.
 \end{align*}
 Moreover let  $b\in C^1(\overline{\Omega}\times(0,T),\R^n)\cap L_{\text{loc}}^{\infty}(\overline{\Omega}\times[0,T),\R^n)$ fulfil $b\cdot \nu=0$ on  $\partial\Omega\times(0,T)$. If then $\bar{u}$ and $\ubar{u}$ are an upper and a lower solution, respectively, to the equation 
 \begin{align*}
  u_t=Δ u+\nabla\cdot(b(x,t) u(x,t))+f(x,t,u(x,t)), x\in\Omega, t\in (0,T),
 \end{align*}
 i.e. if $\bar{u}$ and $\ubar{u}$ belong to $C^0(\overline{\Omega}\times[0,T))\cap C^{2,1}(\Omega\times(0,T))$ and satisfy 
 \begin{align*}
  \bar{u}_t\geq Δ \bar{u}+\nabla \cdot(b(x,t)\bar{u}(x,t))+f(x,t,\bar{u}(x,t)) \quad\text{for }(x,t)\in\Omega\times(0,T),\\
  \ubar{u}_t\leq Δ \ubar{u}+\nabla \cdot(b(x,t)\ubar{u}(x,t))+f(x,t,\ubar{u}(x,t)) \quad\text{for }(x,t)\in\Omega\times(0,T),
 \end{align*}
 and if 
 \begin{align*}
  \partial_\nu \bar{u}\geq \partial_\nu \ubar{u} \quad\text{on }\partial\Omega\times(0,T)
 \end{align*}
and
\begin{align*}
  \bar{u}(x,0)\geq\ubar{u}(x,0) \quad\text{for all } x \in\Omega,
 \end{align*}
 are fulfilled, then 
 \begin{align*}
  \bar{u}(x,t)\geq\ubar{u}(x,t) \quad\text{for all } x\in\Omega,~t\in[0,T).
 \end{align*}
\end{satz}

\begin{proof}
 Let $ζ\in C^{\infty}(\R)$ be such that $ζ'\geq 0$ on $\R$, $ζ\equiv 0$ on $(-\infty,\tel{2}]$ and $ζ\equiv 1$ on $[1,\infty)$. For $\delta\in (0,1)$ define  
 \begin{align*}
  \phi_{\delta}(\xi):=\int_0^{\xi}ζ\left(\frac{\sigma}{\delta}\right)\diff \sigma, ~\xi\in\R.
 \end{align*}
 Then $\phi_{\delta}\in C^{\infty}(\R)$ is nonnegative with $\phi_{\delta}'\geq 0$, $\phi_{\delta}''\geq 0$, $\phi_{\delta}'\equiv 1$ on  $[\delta, \infty)$ and $\phi_{\delta}\equiv 0$ on $(-\infty,\frac{\delta}{2})$. In particular, $\phi_{\delta}(\xi)\lvert \xi\rvert=\phi_{\delta}(\xi)\xi$. 
 Moreover, as $\delta \searrow 0$, we have 
 \begin{align*}
  \phi_{\delta}'(\xi)\nearrowζ_{(0,\infty)}(\xi)\quad\text{as well as }\phi_{\delta}(\xi)\nearrow\xi_{+}\quad\text{for all }\xi\in\R,
 \end{align*}
 where $\xi_+:=\max\set{\xi, 0}$. 
 For arbitrary $\hat{T}\in(0,T)$ we define $K:=\ubar{u}(\overline{\Omega}\times[0,\hat{T}])\cup\bar{u}(\overline{\Omega}\times[0,\hat{T}])$. Due to the local Lipschitz continuity of $f$ with respect to $u$, there is $L(\hat{T})>0$ satisfying  
 \begin{align}\label{lip}
  |f(x,t,r)-f(x,t,s)|\leq L(\hat{T})|r-s| \quad\text{for all }r,s\in K, (x,t)\in\overline{\Omega}\times[0,\hat{T}).
 \end{align}
 Furthermore, there is $C(\hat{T})>0$ with $b(x,t)\leq C(\hat{T})$ for $(x,t)\in\Omega\times(0,\hat{T}]$.\\
 In $\Omega\times(0,\hat{T})$, the difference $\ubar{u}-\bar{u}$ obeys  
 \begin{align*}
  \left(\ubar{u}-\bar{u}\right)_t\leq Δ\left(\ubar{u}-\bar{u}\right)+\nabla\cdot(b(x,t)\left(\ubar{u}-\bar{u}\right))+f(x,t,\ubar{u})-f(x,t,\bar{u}).
 \end{align*}
 Multiplication with the nonnegative function $\phi_\delta\left(\ubar{u}-\bar{u}\right)$ and integration over $\Omega$ by (\ref{lip}) leads to 
 \begin{align*}
  \io\phi_\delta\left(\ubar{u}-\bar{u}\right)\left(\ubar{u}-\bar{u}\right)_t\leq& \io \phi_\delta\left(\ubar{u}-\bar{u}\right)Δ\left(\ubar{u}-\bar{u}\right)+\io\phi_\delta\left(\ubar{u}-\bar{u}\right)\nabla\cdot\left(b(x,t)\left(\ubar{u}-\bar{u}\right)\right)\\
  &+L(\hat{T})  \io\underbrace{\phi_\delta\left(\ubar{u}-\bar{u}\right)|\ubar{u}-\bar{u}|}_{=\phi_\delta\left(\ubar{u}-\bar{u}\right)(\ubar{u}-\bar{u})} \quad\text{in }(0,\hat{T}).
 \end{align*}
 Integration by parts shows 
 \begin{align*}
  \io \phi_\delta\left(\ubar{u}-\bar{u}\right)Δ\left(\ubar{u}-\bar{u}\right) &=-\io\nabla(\phi_\delta\left(\ubar{u}-\bar{u}\right))\cdot\nabla\left(\ubar{u}-\bar{u}\right)+\int_{\partial\Omega}\phi_\delta\left(\ubar{u}-\bar{u}\right)(\partial_\nu\ubar{u}-\partial_\nu\bar{u})\\
  &\leq -\io\phi_\delta'\left(\ubar{u}-\bar{u}\right)|\nabla\left(\ubar{u}-\bar{u}\right)|^2\quad\text{in }(0,\hat{T}),
 \end{align*}
 because $\phi_\delta\geq 0$ and $\partial_\nu\ubar{u}-\partial_\nu\bar{u}\leq 0$ in $\partial\Omega$.\\
Additionally, integration by parts, aided by Young's inequality, $b\cdot\nu=0$ on $\partial\Omega$ and the estimate for $b$, has us conclude 
 \begin{align*}
  \io &\phi_\delta\left(\ubar{u}-\bar{u}\right)\nabla\cdot(b(x,t)\left(\ubar{u}-\bar{u}\right))\\
  &= -\io\phi_\delta'\left(\ubar{u}-\bar{u}\right)\nabla\left(\ubar{u}-\bar{u}\right)b(x,t)\left(\ubar{u}-\bar{u}\right) +\int_{\partial\Omega}\phi_\delta\left(\ubar{u}-\bar{u}\right)\left(\ubar{u}-\bar{u}\right)b(x,t)\cdot\nu\\
  &\leq\tel{2}\io|\nabla\left(\ubar{u}-\bar{u}\right)|^2\phi_\delta'\left(\ubar{u}-\bar{u}\right)+ \frac{C(\hat{T})^2}{2}\io \phi_\delta'\left(\ubar{u}-\bar{u}\right)\left(\ubar{u}-\bar{u}\right)^2 \quad\text{in }(0,\hat{T}).
 \end{align*}
Together, in $(0,\hat{T})$ this ensures 
 \begin{align*}
  \io \phi_\delta\left(\ubar{u}-\bar{u}\right)\left(\ubar{u}-\bar{u}\right)_t\leq& -\tel{2}\io |\nabla\left(\ubar{u}-\bar{u}\right)|^2\phi_\delta'\left(\ubar{u}-\bar{u}\right)\\
  &+ \frac{C(\hat{T})^2}{2}\io \phi_\delta'\left(\ubar{u}-\bar{u}\right)\left(\ubar{u}-\bar{u}\right)^2 
  +L(\hat{T})\io \phi_\delta\left(\ubar{u}-\bar{u}\right)\left(\ubar{u}-\bar{u}\right)\\
  \leq &\;\, \frac{C(\hat{T})^2}{2}\io \phi_\delta'\left(\ubar{u}-\bar{u}\right)\left(\ubar{u}-\bar{u}\right)^2+L(\hat{T})\io \phi_\delta\left(\ubar{u}-\bar{u}\right)\left(\ubar{u}-\bar{u}\right)
 \end{align*}
 and by integration over $(0,t)$ for any $t\in(0,\hat{T}]$ we arrive at 
 \begin{align}\label{eq:vs}
  \int_0^t \io \phi_\delta\left(\ubar{u}-\bar{u}\right)\left(\ubar{u}-\bar{u}\right)_t 
  &\leq \frac{C(\hat{T})^2}{2}\int_0^t\!\!\io \phi_\delta'\left(\ubar{u}-\bar{u}\right)\left(\ubar{u}-\bar{u}\right)^2+L(\hat{T})\int_0^t\!\!\io \phi_\delta\left(\ubar{u}-\bar{u}\right)\left(\ubar{u}-\bar{u}\right).
 \end{align}
We now define 
\begin{align*}
  \Phi_{\delta}(y):=\int_0^y\phi_\delta(s)\diff s
 \end{align*}
 for $y\in\R$. Because of $\phi_{\delta}(s)\nearrow s_+$ and monotonicity of the integral, Beppo-Levi's theorem in the limit  $\delta\searrow 0$ shows 
 \begin{align}\label{eq:vs_phi}
  \Phi_{\delta}(y)\nearrow\int_0^y s_+\diff s=\tel{2}y_+^2.
 \end{align}
 With Fubini's theorem and the substitution $y(\cdot,t)=\uunten(\cdot,t)-\uoben(\cdot,t)$ we can rewrite the left-hand side of  (\ref{eq:vs}) as 
 \begin{align*}
   \int_0^t &\io \phi_\delta\left(\ubar{u}-\bar{u}\right)\left(\ubar{u}-\bar{u}\right)_t=\io \int_0^{\uunten(\cdot,t)-\uoben(\cdot,t)} \phi(y)\diff y=\io \Phi_{\delta}(\uunten(\cdot,t)-\uoben(\cdot,t)),
 \end{align*}
 for $t\in(0,\hat{T})$, where by (\ref{eq:vs_phi}) and Beppo-Levi's theorem for all $t\in(0,\hat{T})$ we have 
 \begin{align*}
  \io \Phi_{\delta}(\uunten(\cdot,t)-\uoben(\cdot,t))\to \tel{2}\io \left(\uunten(\cdot,t)-\uoben(\cdot,t)\right)_+^2 \quad\text{as }\delta \searrow 0.
 \end{align*}
Due to 
\begin{align*}
  \phi_\delta\left(\ubar{u}-\bar{u}\right)\nearrow\left(\ubar{u}-\bar{u}\right)_+ \quad\text{and } \phi_\delta'\left(\ubar{u}-\bar{u}\right)\nearrowζ_{(0,\infty)}\left(\ubar{u}-\bar{u}\right)\quad\text{as }\delta\searrow0,
 \end{align*}
i.e.
\begin{align*}
  \phi_\delta'\left(\ubar{u}-\bar{u}\right)(\uunten-\uoben)\nearrow (\uunten-\uoben)_+\quad\text{as }\delta\searrow0,
 \end{align*}
another application of Beppo-Levi's theorem, this time on the right-hand side of (\ref{eq:vs}), shows that in the limit  $\delta\searrow 0$
 \begin{align*}
  \frac{C(\hat{T})^2}{2}\int_0^t\!\!\io \phi_\delta'\left(\ubar{u}-\bar{u}\right)\left(\ubar{u}-\bar{u}\right)^2&+L(\hat{T})\int_0^t\!\!\io \phi_\delta\left(\ubar{u}-\bar{u}\right)\left(\ubar{u}-\bar{u}\right)\\
  \to \frac{C(\hat{T})^2}{2}\int_0^t\!\!\io \left(\ubar{u}-\bar{u}\right)_+^2&+L(\hat{T})\int_0^t\!\!\io \left(\ubar{u}-\bar{u}\right)_+^2.
 \end{align*}
 In conclusion, from (\ref{eq:vs}), for $t\in(0,\hat{T})$ we have obtained 
 \begin{align*}
  \tel{2}\io (\uunten(\cdot,t)-\uoben(\cdot,t))_+^2\leq \frac{C(\hat{T})^2+2L(\hat{T})}{2}\int_0^t\!\!\io \left(\uunten-\uoben\right)_+^2.
 \end{align*}
An application of Grönwall's inequality %with $y=\io(\uunten-\uoben)_+^2$, $\alpha=0$ and $\beta(s)\equiv C(\hat{T})^2+2L(\hat{T})$
 yields 
\(% \begin{align*}
  \io (\uunten-\uoben)_+^2\leq 0\) on $(0,\hat{T})$ % \quad\text{on }(0,\hat{T}),
% \end{align*}
 and hence $(\uunten-\uoben)_+\equiv 0$, that is, $\ubar{u}\leq\bar{u}$. Since $\hat{T}<T$ was arbitrary, we may conclude  $\ubar{u}\leq\bar{u}$ in $\overline{\Omega}\times[0,T)$.
%  Ab hier alt
%  \begin{align*}
%   \tel{2}\dt\io\left(\ubar{u}-\bar{u}\right)_+^2=\io \left(\ubar{u}-\bar{u}\right)_+\left(\ubar{u}-\bar{u}\right)_t\leq  \tel{2}\io C(\hat{T})^2\left(\ubar{u}-\bar{u}\right)_+^2 +L(\hat{T})\io \left(\ubar{u}-\bar{u}\right)_+^2,
%  \end{align*}
%  also
%  \begin{align*}
%   \dt\io\left(\ubar{u}-\bar{u}\right)_+^2\leq (C(\hat{T})+2L(\hat{T}))\io \left(\ubar{u}-\bar{u}\right)_+^2.
%  \end{align*}
%  Aus dieser einfachen Differentialungleichung for $\io\left(\ubar{u}-\bar{u}\right)_+^2$ folgt 
%  \begin{align*}
%   \io\left(\ubar{u}-\bar{u}\right)_+^2\leq \left(\io(\ubar{u}(\cdot,0)-\bar{u}(\cdot,0))_+^2\right)\euler^{(C+2L)t}.
%  \end{align*}
%  Da nach Voraussetzung $\ubar{u}(\cdot,0)\leq\bar{u}(\cdot,0)$ and somit $(\ubar{u}(\cdot,0)-\bar{u}(\cdot,0))_+\equiv0$ ist, folgt hieraus $\left(\ubar{u}-\bar{u}\right)_+\equiv0$ in $\overline{\Omega}\times[0,\hat{T}]$,  
\end{proof}

\begin{bemerkung}\label{bem:vs}
 It is not immediately clear that the comparison theorem \ref{vs} can be applied to the system under consideration, because for the first equation it is not a priori known whether 
 \begin{align*}
 b\in C^1(\overline{\Omega}\times(0,T),\R^n)\cap L_{\text{loc}}^{\infty}(\overline{\Omega}\times[0,T),\R^n).
 \end{align*}
 The following considerations, however, will ensure its applicability and we have hence used Theorem \ref{vs} throughout the article without repeating these arguments explicitly. 
 At first, the equation 
 \begin{align*}
  v_t=Δ v -uv \quad\text{on  }\Omega\times(0,\Tmax)
 \end{align*}
 is studied. Here we interpret $u$ as given function from Theorem \ref{satz:lok} (or, in the proof of Theorem \ref{satz:lok}, from Lemma \ref{satz:survey}), i.e.  $u\in C^0(\overline{\Omega}\times[0,\Tmax))\cap C^{2,1}(\overline{\Omega}\times(0,\Tmax))$. With $b\equiv 0$ and $f(x,t,v):=-u(x,t)v$ the conditions of the comparison theorem are fulfilled due to continuity of $u$. For given $T<\Tmax$ one can find $C>0$ such that $u\leq C$ on $\overline{\Omega}\times[0,T]$, because $u$ is continuous. Hence  $\ubar{v}=(\inf_{x\in\Omega}v_0(x))\euler^{-Ct}$ is a subsolution of the equation, since 
 \begin{align*}
  \ubar{v}_t=-C\ubar{v}=Δ \ubar{v}-C\ubar{v}\leq Δ \ubar{v}-u\ubar{v}.
 \end{align*}
Furthermore we have 
 \begin{align*}
  \partial_{\nu} \ubar{v}=0 \quad\text{on }\partial\Omega\times(0,T) \quad\text{on } \ubar{v}(\cdot,0)\leq v_0 \quad\text{on  }\Omega\times(0,T).
 \end{align*}
 Therefore by comparison with $\bar{v}=v$ it follows that $v(x,t)\geq (\inf_{x\in\Omega}v_0(x))\euler^{-Ct}$ on $\Omega\times(0,T)$.\\
 For $T<\Tmax$ we thus can consider the first equation 
 \begin{align*}
  u_t=Δ u-\chi\kreuz+f(u) \quad\text{in }\Omega\times(0,T)
 \end{align*}
with the given function $v$. Here we now have  
% \begin{align*}
\( 
 f(x,t,u)=f(u)%\kappa u-\mu u^2,
\) 
% \end{align*}
with the differentiable function $f$, which apparently is continuous and locally Lipschitz continuous with respect to $u$, and 
 \begin{align*}
 b(x,t)=\chi\frac{\nabla v(x,t)}{v(x,t)}.
 \end{align*}
 Due to $v\geq (\inf_{x\in\Omega}v_0(x))\euler^{-CT}>0$, $b$ is well-defined. 
Since furthermore $v\in C^{2,1}(\overline{\Omega}\times(0,\Tmax))$, even $b\in C^1(\overline{\Omega}\times(0,T),\R^n)$. As  $\partial_\nu v=0$ on $\partial\Omega$, it follows that $b\cdot \nu=0$ on $\partial\Omega$. Due to continuity of $u$ and $v$ on $\overline{\Omega}\times[0,T]$, there is $M>0$ satisfying $\norm{\infty}{uv}\leq M$. Correspondingly, for any $\hat{t}\leq T$,   
 \begin{align*}
  \norm{\infty}{\frac{\nabla v(\cdot,t)}{v(\cdot,t)}}\leq (\inf v_0)^{-1}\euler^{Ct}\norm{\infty}{\nabla v(\cdot,t)}\quad\text{for all } t\in[0,\hat{t}],
 \end{align*}
 where, due to the variation-of-constants formula and semigroup estimates (\cite[Lemma 1.3 (ii)]{lplq}), 
 \begin{align*}
  \norm{\infty}{\nabla v(\cdot,t)}\leq &\norm{\infty}{\nabla \euler^{tΔ}v_0}+\int_0^t\norm{\infty}{\nabla \euler^{(t-s)Δ}(uv)(\cdot,s)}\diff s\\
  \leq & c_1 \norm{\infty}{\nabla v_0}+c_2 \int_0^t\left(1+(t-s)^{-\tel{2}}\right)\norm{\infty}{(uv)(\cdot,s)}\diff s \\
  \leq & c_1 \norm{\infty}{\nabla v_0}+c_2 M (\hat{t}+2\hat{t}^{\tel{2}})=:K<\infty \quad\text{for all } t\in[0,\hat{t}],
 \end{align*}
 because $v_0\in W^{1,\infty}(\Omega)$.\\
 Thereby, indeed, $b\in C^1(\overline{\Omega}\times(0,T),\R^n)\cap L_{\text{loc}}^{\infty}(\overline{\Omega}\times[0,T),\R^n)$, and the comparison theorem becomes applicable to suitable sub- and supersolutions of this equation.
\end{bemerkung}

\end{appendix}

{\footnotesize
\bibliographystyle{abbrv}
%\bibliography{quellen.bib}

}
\end{document}